\documentclass[a4paper,11pt]{article}
\usepackage[T1]{fontenc}
\usepackage[latin1]{inputenc}
\usepackage{amsmath,amssymb,euscript}
\usepackage{amsthm}
\usepackage{bbm}
\usepackage{graphicx}
\usepackage{epic}
\usepackage{epsfig}
\usepackage{pstricks}
\usepackage{psfrag}
 \usepackage{curves}
\usepackage{mathrsfs}
\usepackage{rotating}
\usepackage{color}
\usepackage{enumerate}

\newtheorem{thm}{Theorem}[section]

\newtheorem{prop}[thm]{Proposition}
\newtheorem{rem}[thm]{Remark}

\def\be{\begin{eqnarray}}
\def\ee{\end{eqnarray}}
\def\ben{\begin{eqnarray*}}
\def\een{\end{eqnarray*}}

\numberwithin{equation}{section}
\numberwithin{figure}{section}

\def\be{\begin{eqnarray}}
\def\ee{\end{eqnarray}}

\def\me{\medskip\noindent}
\def\bi{\bigskip\noindent}
\newcommand{\Sgn}{\mbox{Sgn}}

\newcommand{\Co}{\mathcal{C}}

\def\D{\mathbb{D}}

\def\N{\mathbb{N}}

\def\R{\mathbb{R}}
\def\E{\mathbb{E}}
\def\X{\mathcal{X}}

\title{\bf Stochastic dynamics for adaptation and evolution of microorganisms}

\author{Sylvain Billiard\thanks{Univ. Lille, CNRS, UMR 8198 - Evo-Eco-Paleo, F-59000 Lille, France; E-mail: \texttt{sylvain.billiard@univ-lille1.fr}}, \quad Pierre Collet\thanks{CPHT, Ecole Polytechnique, CNRS, route de
    Saclay, 91128 Palaiseau Cedex-France; E-mail: \texttt{collet@cpht.polytechnique.fr}}, \quad R\'egis Ferri\`ere\thanks{Eco-Evolution Math\'ematique,
    CNRS
    UMR 7625, Ecole Normale Sup\'erieure, 46 rue d'Ulm, 75230 Paris, France; E-mail: \texttt{ferriere@biologie.ens.fr}}, \quad  Sylvie M\'el\'eard\thanks{CMAP, Ecole Polytechnique, CNRS, route de
    Saclay, 91128 Palaiseau Cedex-France; E-mail: \texttt{sylvie.meleard@polytechnique.edu}}, \quad Viet Chi Tran\thanks{Univ. Lille, CNRS, UMR 8524 - Laboratoire Paul Painlev\'e, F-59000 Lille, France; E-mail: \texttt{chi.tran@math.univ-lille1.fr}}}

\date{\today}

\begin{document}

\maketitle

\begin{abstract}
 We present a model for the dynamics of a population of bacteria with a continuum of traits, who compete for resources and exchange horizontally (transfer) an otherwise vertically inherited trait with possible mutations. Competition influences individual demographics, affecting population size, which feeds back on the dynamics of transfer. We consider a stochastic individual-based pure jump process taking values in the space of point measures, and whose jump events describe the individual reproduction, transfer and death mechanisms.  In a large population scale, the stochastic process is proved to converge to the solution of a nonlinear integro-differential equation. When there are only two different traits and no mutation, this equation reduces to a non-standard two-dimensional dynamical system. We show how crucial the forms of the transfer rates are for the long-term behavior of its solutions. We describe the dynamics of invasion and fixation when one of the two traits is initially rare, and compute the invasion probabilities. Then, we study the process under the assumption of rare mutations. We prove that the stochastic process at the mutation time scale converges to a jump process which describes   the successive invasions of successful mutants. We show that  the horizontal transfer can have a major impact on the distribution of the successive mutational fixations, leading to dramatically different behaviors, from expected evolution scenarios to evolutionary suicide.  Simulations are given to illustrate these phenomena.
\end{abstract}

\me Keywords:
horizontal gene transfer, bacterial conjugation, stochastic individual-based models, long time behavior, large population approximation, interactions, fixation probability, trait substitution sequence, adaptive dynamics, canonical equation.

\bigskip

\bigskip

\section{Introduction and  biological context}
A distinctive signature of living systems is Darwinian evolution, that is, a propensity to generate as well as self-select individual
diversity. To capture this essential feature of life while describing the dynamics of populations, mathematical models must be rooted in
the microscopic, stochastic description of discrete individuals characterized by one or several adaptive traits and interacting with each
other.
In this paper, we focus on  the mathematical  modeling of bacteria evolution, whose understanding is fundamental in biology, medicine and industry.
The  ability of a bacteria  to survive and  reproduce depends on  its genes,
and  evolution  mainly  results from  the following basic mechanisms:
heredity, i.e. transmission of the ancestral traits to offspring (also called vertical transmission);
   mutation which occurs during  vertical transmission and generates variability of the traits; selection which results from the interaction between individuals and their environment; exchange of genetic information between non-parental individuals during their lifetimes (also called horizontal gene transfer (HGT)).
   In many biological situations, competition between individuals and vertical and horizontal transfers are involved. The combined resulting effects may have a key role in the transmission of an epidemic, in the  development of antibiotic resistances, in epigenetics or for the bacterial degradation of novel compounds such as human-created pesticides.
 There are several mechanisms for horizontal gene transfer:
  transformation, where some DNA filaments directly enter the cell from the surrounding environment;
 transduction, where  DNA is carried and introduced into the cell by viruses (phages); and
  conjugation, when circular DNA (plasmids) replicates into cells and is transmitted from a cell to another one, independently of the chromosome. Conjugation plays a main role for infectious diseases since the genes responsible for virulence or antibiotic resistance are usually carried by plasmids.  In this paper, we  focus on conjugation modeling in order to understand transmission of pathogens and
 the evolution of antibiotic resistances.

 \me We
propose  a general stochastic eco-evolutionary model of population dynamics with horizontal and vertical genetic transmissions. The stochastic process describes a finite population of
discrete interacting individuals characterized by one or
several adaptive phenotypic traits, in the vein of the models developed in \cite{fourniermeleard}.
Other models for HGT have been proposed in the literature, based on the seminal contribution of Anderson and May on host-pathogen deterministic population dynamics \cite{andersonmay1979} (see also \cite{levinetal1979,stewartlevin1977}) or on a population genetics framework without  ecological concern (see \cite{baumdickerpfaffelhuber,novozhilovetal2005,tazzymanbonhoffer}). Additionally, the previous models assume unilateral transfer, dividing the population into two classes: donors and recipients.   In the present paper, we relax most of the previous limitations.

 The stochastic model is a continuous time pure jump process with values in the space of point measures. The jump events are
 births with or without mutation,   horizontal transfers and deaths, with dynamics depending on the trait values of each individual and on the global population.
 Our model covers both cases of frequency- and density-dependent horizontal transfer rates; these dependencies appear as special cases of a more general form of transfer rate, that we call Beddington-DeAngelis by analogy with a similar model used to describe predator-prey contacts (\cite{beddington1975, deangelisetal1975}).\\
 In a large population limit, using ideas developed in Fournier and M\'el\'eard \cite{fourniermeleard}, the stochastic process is shown to converge to the solution of a nonlinear integro-differential equation whose existence and uniqueness are proved. (See also Billiard et al. \cite{billiardcolletferrieremeleardtran} for different evolutionary behaviors
depending on the order of magnitude of
population size, mutation probability and mutation step
size).
In the case where the trait support is composed of two values, the equation reduces to a non-standard two-dimensional dynamical system  whose long time behavior is studied. This study highlights the impact of HGT on the maintenance of polymorphism and the invasion or elimination of pathogen strains.
When a trait is initially rare in the population (e.g. a mutation of the common trait),  its subpopulation is purely stochastic and we explain how HGT can drastically influence its probability of invasion and time to fixation. To do so, we combine  the stochastic behavior of the mutant population size with  the deterministic approximation of the resident population size. Then we assume that  mutations are rare enough to imply a separation between the competition and mutation time scales, following ideas of Champagnat et al. \cite{champagnat06} in a case without HGT. Here, under an Invasion-Implies-Fixation assumption, a pure jump process is derived from the population size process at the mutation time scale, for which the jump measure is strongly affected by the horizontal transfer. In the last section we present simulations in a case of unilateral transfer, which highlight the effect of HGT on evolution. In particular, we show that HGT can completely change the evolutionary outcomes. Depending on the transfer rate, we can obtain dramatically different behaviors, from expected evolution scenarios to evolutionary suicide.

\section{A general stochastic individual-based model for vertical and horizontal trait  transmission}

\subsection{The model}

We model a bacteria
population with a stochastic system of interacting
individuals (cf. Fournier-M\'el\'eard \cite{fourniermeleard}, Champagnat-Ferri\`ere-M\'el\'eard \cite{champagnatferrieremeleard2, champagnatferrieremeleard}). Each individual is characterized by a quantitative parameter $x$, called trait,  which belongs to a compact subset $\,{\cal X}\,$ of $\, \R^d$ and summarizes the phenotype or genotype of the individual. The trait determines the  demographic rates. It is inherited from parent to offspring (reproduction is asexual), except when a mutation occurs, in which case the trait of the offspring takes a new value. It can also be transmitted by horizontal transfer from an individual to another one.   The demographic and ecological rates are scaled by  $K$ which is taken as a measure of the "system size" (resource limitation, living area, carrying capacity,  initial number of individuals).
We will derive
macroscopic models from the individual process by letting
$\,K\,$ tend to infinity  with the appropriate renormalization ${1\over K}$ for individuals' weight.

\me  At each time $t$, the population is described by the point measure
$$\nu_t^K(dx)=\frac{1}{K}\,\sum_{i=1}^{N^K_t}\, \delta_{X_i(t)}(dx).$$
 $N^K_t= K \int \nu_t^K(dx)\, $ is  the size of the population at time $t$ and  $X_i(t)$ the trait of the $i$-th individual living at $t$, individuals being ranked by lexicographical trait values.

\me Let us now describe the transitions of the measure-valued Markov process $\,(\nu_t^K, t\geq 0)\,$.

\me An individual  with trait $x$ gives birth to a new individual with rate $b_{K}(x)$. With probability $1-p_K$, the new individual carries the trait $x$ and with probability $p_K$, there is a mutation on the trait. The trait of the new individual is $z$  chosen in the probability distribution $m(x,dz)$.

\smallskip  \noindent An individual with trait $x$ dies with intrinsic death rate $d_{K}(x)$ or from the competition with any other individual alive at the same time. If the competitor has  the trait   $y$, the additional death rate is  $ C_{K}(x,y)$. Then in the population $\nu={1\over K}\sum_{i=1}^n \delta_{x_i}$,  the individual death rate due to competition is  $\,\sum_{i=1}^n {C_{K}}(x,x_i) = KC_{K}*\nu(x)$.

\smallskip  \noindent  In addition, individuals can exchange genetic information. Horizontal transfers can occur in both directions: from individuals $x$ to $y$ or the reverse, possibly at different rates. In a population $\nu$, an individual with trait  $x$ chooses a partner with trait $y$ at rate $h_K(x,y,\nu)$.  After transfer, the couple $(x,y)$ becomes $(T_1(x,y),T_2(x,y))$. In the specific case of bacterial conjugation, the recipient $y$ acquires the trait $x$ of the donor (i.e. $(T_{1}(x,y),T_{2}(x,y)) = (x,x)$). This occurs for instance  when the donor transmits a copy of its plasmid to individuals devoid of plasmid (in that case, transfer is unilateral).
 We refer to the paper of Hinow et al. \cite{hinowlefollmagalwebb} for other examples.

\subsection{Generator}
We denote by $\,{\cal M}_{K}(\X)\,$ the set of point measures on $\,{\cal X}$ weighted by $1/K$ and by $M_{F}(\X)$ the set of finite measures on ${\cal X}$.
The generator of the process $(\nu^K_t)_{t\geq 0}$ is given for measurable bounded functions $F$ on $\,{\cal M}_{K}(\X)\,$ and  $\nu={1\over K}\sum_{i=1}^n \delta_{x_i}$ by
\begin{align}
L^KF(\nu)= &\sum_{i=1}^n b_{K}(x_i)(1-p_K) \big(F(\nu +{1\over K}\delta_{x_i}) - F(\nu)\big)\nonumber\\
 +&\sum_{i=1}^n b_{K}(x_i)\, p_K \int_{\cal X} \big(F(\nu +{1\over K}\delta_{z}) - F(\nu)\big) m(x_i, dz)\nonumber\\
 +&\sum_{i=1}^n \big(d_{K}(x_i)+ KC_{K}*\nu(x_i)\big)\big(F(\nu  -{1\over K}\delta_{x_i}) - F(\nu)\big)\nonumber\\
 +&\sum_{i,j=1}^n h_K(x_i,x_j,\nu)\big(F(\nu +{1\over K}\delta_{T_1(x_i,x_j)}+{1\over K}\delta_{T_2(x_i,x_j)} -{1\over K}\delta_{x_i} -{1\over K}\delta_{x_j}) - F(\nu)\big).\label{gen}
\end{align}
In particular, if we consider the function $F_{f}(\nu)=\langle \nu, f\rangle$ for a bounded and measurable function $f$ on ${\cal X}$  and $\nu\in{\cal M}_{K}(\X)$ with the notation $\langle \nu,f\rangle = \int f(x) \nu(dx)$, we have
\begin{align}
L^KF_{f}(\nu)= & \int_{\cal X} \nu(dx) \Big[ b_{K}(x)\Big((1-p_K) f(x) + p_K\int_{\cal X} f(z) m(x,dz)  \Big)\nonumber\\
 & \hspace{1.2cm} - \big(d_{K}(x)+ KC_{K}*\nu(x)\big)f(x)\nonumber\\
 & \hspace{1.2cm} + \int_{\cal X} {K h_K(x,y,\nu)} \Big(f(T_1(x,y))+f(T_2(x,y))-f(x)-f(y)\Big)\nu(dy)
\Big].
\end{align}

\bi Assuming that for any $K$, the functions $b_K$, $d_K$, $K C_K$ and $Kh_{K}$ are bounded,
it is  standard to construct the measure valued process $\,\nu^K\,$ as the solution of a stochastic differential equation driven by Poisson point measures and to derive the following moment and martingale properties (see for example  \cite{fourniermeleard} or Bansaye-M\'el\'eard \cite{bansayemeleard}).

\begin{thm}
 \label{martingales}
  Under the previous assumptions  and assuming that for some $p\geq 2$,

   $\mathbb{E} \left( \langle
      \nu^K_{0},1 \rangle^p \right) <\infty$, we have the following properties.
  \begin{description}
  \item[\textmd{(i)}] For all measurable functions $F$ from ${\cal
      M}_K(\X)$ into $\mathbb{R}$ such that for some constant $C$ and for all
    $\nu \in {\cal M}_K(\X)$, $\vert F(\nu) \vert + \vert L^KF(\nu)
    \vert \leq C (1+\left< \nu,1 \right>^p)$, the process
    \begin{equation}
      \label{pbm2}
     F(\nu^K_t) - F(\nu^K_0) - \int_{0}^t  L^KF (\nu^K_s) ds
    \end{equation}
    is a c\`adl\`ag $({\cal F}_t)_{t\geq 0}$-martingale starting from
    $0$.
  \item[\textmd{(ii)}] For  a  bounded and measurable  function $f$ on ${\cal X}$,
    \begin{align}
&\int f(x) \nu^K_{t }(dx) = \int f(x) \nu^K_{0 }(dx) +    M^{K,f}_t \notag \\
&+\int_{0}^t\int_{\cal X} \bigg\{\Big((1-p_K)b_{K}(x)
      -d_{K}(x) -KC_{K}*\nu^K_{s }(x)\Big)f(x)+p_K b_{K}(x)\int_{\cal X}f(z)\,m(x,dz)  \nonumber\\
&+  \int_{\cal
        X}{K h_K(x,y,\nu^K_s)} \Big(f(T_1(x,y))+f(T_2(x,y))-f(x)-f(y)\Big)\nu^K_s(dy)\bigg\}
   \nu^K_{s }(dx) ds
  , \label{eq:mart}
    \end{align}
   where  $\,M^{K,f}\,$ is a c\`adl\`ag square integrable martingale starting from $0$
    with quadratic variation
    \begin{align}
      \langle M^{K,f}\rangle_t &= {1\over K}\int_{0}^t \int_{\cal
        X}\bigg\{\Big((1-p_K)b_{K}(x)+d_{K}(x)+KC_{K}*\nu^K_s(x)\Big)f^2(x)
      \notag \\ &+p_Kb_{K}(x)\int_{\cal X}f^2(z)\,m(x,dz)\nonumber\\
      & +  \int_{\cal
        X}{K h_K(x,y,\nu^K_s)} \Big(f(T_1(x,y))+f(T_2(x,y))-f(x)-f(y)\Big)^2\nu^K_s(dy)
      \bigg\} \nu^K_s(dx) ds. \label{qv1}
    \end{align}
  \end{description}
\end{thm}

\section{Large population limit and rare mutation in the ecological time-scale}\label{section:ODE}

\subsection{A deterministic approximation}
We derive some macroscopic approximation by letting the scaling parameter $K$ tend to infinity with the additional assumption of rare mutation, i.e.
\be
\label{mutrare}\lim_{K\to \infty} p_K  = 0.\ee The timescale is unchanged. It is  called the `ecological' timescale of births, interactions (competition and transfer), and deaths.

\me The next hypotheses describe the different scalings  considered in the paper. 

\me {\bf Assumptions $(H)$}

 \me {\it (i) When $K$ tends to infinity, the stochastic  initial point measures $\,\nu^{K}_0 $ converge  in probability (and for the weak topology)  to the deterministic measure $ \xi_{0}\in M_{F}({\cal X})\,$ and
 $$\sup_{K} \mathbb{E}(\langle \nu^K_0, 1\rangle^3)<\infty.$$

\me (ii) When $K \rightarrow \infty$, the functions $b_{K}$ and  $d_{K}$ (respectively $K C_{K}$) converge uniformly on ${\cal X}$ (respectively on ${\cal X}\times {\cal X}$) to the continuous functions $b$ and $d$ (respectively to  $C$).

\me (iii)  We assume that for any $x, y\in {\cal X}$,
$$b(x)-d(x) >0\,, \,C(x,y)>0.$$
This means that in absence of competition, the subpopulation with trait $x$ is super-critical and that the regulation of the population size comes from the competition. We  denote  by
$$r(x)=b(x)-d(x)$$
the intrinsic  growth rate of the subpopulation of trait $x$.

\me (iv)  When $K \rightarrow \infty$, the functions   $K h_K$  converge uniformly on ${\cal X}\times {\cal X}$ to a continuous function $\,h$. This function depends on the mechanism of transfer. More precisely, we assume
\begin{equation}
K h_K(x,y,\nu) \to h(x,y,\nu)= \frac{\tau(x,y)}{\beta + \mu\,\langle \nu, 1\rangle},\label{tauxlimite}
\end{equation}
where $\tau$ is continuous on ${\cal X}\times {\cal X}$.}

\me \begin{rem}The form \eqref{tauxlimite} is derived from the so-called ``Beddington-DeAngelis'' functional response in the ecological literature (\cite{beddington1975, deangelisetal1975}). This function covers different interesting cases regarding HGT. The horizontal transfer rate for an individual with trait $x$ in the population $\nu$ is $\,\int h(x,y,\nu) \nu(dy) = \langle \nu, \tau(x,.)\rangle/(\beta + \mu\,\langle \nu, 1\rangle)$. Assuming $\mu=0$ or $\langle \nu, 1\rangle$ very small gives a density-dependent HGT rate (denoted DD): the individual transfer rate is proportional to the density of the recipients in the population. Assuming $\beta=0$ or $\langle \nu, 1\rangle$ very large gives a frequency-dependent HGT rate (denoted FD): the individual transfer rate is proportional to the frequency of the recipients. Finally, assuming $\beta \neq 0$ and $\mu \neq 0$ gives a mixed HGT rate between frequency and density-dependent HGT rates (denoted BDA). This general case models some experimental observations for plasmids, for which a correlation between the form (density- versus frequency-dependent) of the transfer rate and the size of the population (low size versus close to carrying capacity)  was suggested (Raul Fernandez-Lopez, pers. com.). We will show that the choice of  $h(x,y,\nu)$ has important consequences on the dynamics.\end{rem}

\bi
\begin{prop}
\label{largepopulation} Assume $(H)$ and \eqref{mutrare}.
Let $T>0$. When $K\rightarrow \infty$, the sequence $(\nu^K)_{K\geq 1}$ converges in probability in $\D([0,T], M_F({\cal X}))$ to the deterministic function $\xi\in \Co([0,T], M_F({\cal X}))$ defined for any continuous function  $f$ as unique solution   of
\begin{align}
\langle \xi_t,f\rangle = & \langle \xi_0,f\rangle +\int_0^t \int_{\cal X} \big(r(x)-C*\xi(x)\big)f(x)\xi_s(dx)\ ds\nonumber\\
+ &  \int_0^t \int_{{\cal X}\times {\cal X}} \Big(f(T_1(x,y))+f(T_2(x,y))-f(x)-f(y)\Big)\frac{\tau(x,y)}{\beta +\mu\,  \langle \xi_s,1\rangle}\xi_s(dy)\xi_s(dx) \ ds. \label{edp1}
\end{align}
\end{prop}
Let us note (by choosing $f\equiv 1$) that the total size $\langle \xi,1\rangle $ of the population  satisfies the  equation
\begin{equation}\label{eq:size_xi}\langle \xi_t,1\rangle = \langle \xi_0,1\rangle + \int_0^t \int_{\cal X} \big(r(x)-C*\xi(x)\big) \xi_s(dx)\ ds.\end{equation}
This equation is not closed and cannot be easily  resolved  except when $C$ is a constant function. Nevertheless by Assumptions $(H)$, the functions $r$ and $C$ are bounded above and below by positive constants on $\cal X$. Then  the process $\langle \xi_t,1\rangle$ is bounded above and below by the solutions of two logistic equations which converge to strictly positive limits when $t\rightarrow \infty$. For example, $\forall t\in \R_+,\ \langle \xi_t,1\rangle \geq \underline{n}_t$ where
$$\frac{d\underline{n}_t}{dt}=\underline{r}\ \underline{n}_t - \bar{C} \underline{n}_t^2,$$with the notation $\underline{r}=\min_{x\in \cal X}r(x)$ and $\bar{C}=\max_{x,y\in \cal X}C(x,y)$.

\bi
\begin{proof} The proof is standard and consists in a tightness and uniqueness argument. The reader will follow the steps detailed in \cite{fourniermeleard} or in \cite{bansayemeleard}: uniform moment estimates on finite time interval, tightness of the sequence of laws, continuity of the limiting values, identification of the limiting values as solutions of \eqref{edp1}, uniqueness of  the solution of \eqref{edp1}. The last point deserves attention.
Let us  consider $(\xi_t^1)_{t\in[0,T]}$ and $(\xi_t^2)_{t\in [0,T]}$  two continuous solutions of  \eqref{edp1} with the same initial condition $\xi_0$. From the remark after \eqref{eq:size_xi}, we have  that \be\label{sup}\,\bar{A}_T =\sup_{t\in[0,T]}\langle\xi_t^1+\xi_t^2,1\rangle<\infty \ \hbox{ and } \  \underline{A}_T=\min\big(\inf_{t\in [0,T]}\langle \xi_t^1,1\rangle, \inf_{t\in [0,T]}\langle \xi_t^1,1\rangle\big)>0.\ee
Let $f$ be a real bounded measurable function on ${\cal X}$ such that $\|f\|_{\infty}\leq 1$. We obtain
\ben
\label{dif}
&&\langle\xi^{1}_{t}-\xi^{2}_{t},f\rangle=\int_0^t \int_{\cal X}\Big( \big(r(x)-C*\xi^1_s(x)\big)f(x)(\xi^1_s -\xi^{2}_{s})(dx) - C*(\xi^1_s-\xi^2_s)(x)\,f(x)\,\xi^2_s(dx)\Big)\, ds\notag\\
&+ &  \int_0^t \int_{{\cal X}\times {\cal X}} \Big(f(T_1(x,y))+f(T_2(x,y))-f(x)-f(y)\Big)\,\tau(x,y)\,\bigg(\frac{1}{\beta +\mu\,  \langle \xi^1_s,1\rangle}(\xi^1_s -\xi^{2}_{s})(dy)\xi^1_s(dx) \notag\\
& &\hskip 2cm +  \frac{1}{\beta +\mu\,  \langle \xi^2_s,1\rangle}\xi^2_s(dy)(\xi^1_s-\xi^{2}_{s})(dx) +\Big(\frac{1}{\beta + \mu  \langle \xi^2_s,1\rangle} -  \frac{1}{\beta + \mu  \langle \xi^1_s,1\rangle}\Big)  \xi^{2}_{s}(dy)\xi^1_s(dx)\bigg)\,ds.
\een
By an elementary computation using Assumptions $(H)$ and \eqref{sup},
we obtain that for any $t\in [0,T]$,
\begin{equation}\label{bpsi}
|\langle
\xi^{1}_{t}-\xi^{2}_{t},f\rangle|\leq  C(T)
\int_{0}^{t}\|\xi^{1}_{s}-\xi^{2}_{s}\|_{TV}\,ds,
\end{equation}
where $C(T)$ is a positive constant. So
taking the supremum over all functions $f$ such that
$\|f\|_{\infty}\leq1$ and applying Gronwall's Lemma we conclude
that for all $t\in[0,T]$
\begin{equation}
\|\xi^{1}_{t}-\xi^{2}_{t}\|_{TV}=0.
\end{equation}
Therefore uniqueness holds for \eqref{edp1}.
\end{proof}

\subsection{Trait replacement and the bacteria conjugation subcase}

We now emphasize on the case where horizontal transmission results from the replacement of the recepient's trait by the donor's trait, i.e. $T_1(x,y)=x$ and $T_2(x,y)=x$.

\me In this case, \eqref{edp1} becomes:
\begin{align}
\langle \xi_t,f\rangle = & \langle \xi_0,f\rangle +\int_0^t \int_{\cal X} \big(r(x)-C*\xi(x)\big)f(x)\xi_s(dx)\ ds\notag\\
&\hskip 2cm +   \int_0^t \int_{{\cal X}\times {\cal X}} f(x)\, \frac{\tau(x,y)-\tau(y,x)}{\beta +\mu\,  \langle \xi_s,1\rangle}\,\xi_s(dy)\xi_s(dx) \, ds. \label{edp2}
\end{align}

\me We note that the behavior of the deterministic dynamical system is influenced by
 HGT only through the `horizontal flux' rate $$\alpha(x,y)=\tau(x,y)-\tau(y,x).$$
  In  Section \ref{invasion}, we will show that in contrast, the fully stochastic population process depends not only on the flux $\alpha$ but also on  $\tau$ itself.

   \noindent The horizontal flux rate $\alpha$ quantifies the asymmetry between transfers in either directions and can be positive as well as negative (or zero in the case of perfectly symmetrical transfer). Note that bacteria conjugation is a subcase: a plasmid is transferred from the plasmid bearer $x$ to the empty individual $y$, while the reverse is not possible (emptiness can not be transferred). This corresponds to the case where  $T_1(x,y)=x$ and $T_2(x,y)=x$ and $\tau(y,x)=0$.

\me
\begin{prop}
Assume that the initial measure $\,\xi_0\,$ is absolutely continuous with respect to the Lebesgue measure, then this property propagates in time and for any $t>0$, $\,\xi_t(dx)= u(t,x)dx$, where $u$ is a  weak solution of the integro-differential equation
\be
\label{det}\, \partial_t u(t,x) = \big(r(x)-C*u(t,x)\big) u(t,x)  +{u(t,x)\over \beta +\mu  \|u(t,.)\|_{1}}\int_{{\cal X}} \alpha(x,y) u(t,y) dy,
\ee
with $\,C*u(t,x)= \int C(x,y)u(t,y)dy$.
\end{prop}

\me Let us mention that at our knowledge, the long time behavior of a solution of this equation is unknown, except in the case without transfer where it  has been studied by Desvillettes et al.  \cite{desvillettes2008}.  Some close equations with transfer have also been considered and studied in the long time by Hinow et al. \cite{hinowlefollmagalwebb} and by Magal-Raoul \cite{magalraoul}.

\bi \section{The two traits case}
\label{section:ODE2}

 \subsection{The dynamical system}\label{sec:LV}

 Let us now assume that the    population  is dimorphic and composed of only two subpopulations characterized by the traits $x$ and $y$.
 We set    ${\cal X}=\{x,y\}$ and define
$\,N^{x,K}_t =  \nu^K_t(\{x\}) \ ;\ N^{y,K}_t =  \nu^K_t(\{y\}) .$ Let us assume that $\,(N^{x,K}_0, N^{y,K}_0)\,$ converges in probability to the deterministic vector $\,(n^x_0, n^y_0)$. Then Proposition \ref{largepopulation} is stated as follows.

\me
\begin{prop}
\label{largepopulationdiploid}
When $\,K \rightarrow \infty\,$, the stochastic process $\,(N^{x,K}_t, N^{y,K}_t)_{t\geq 0}\,$ converges in probability to the solution $\,(n^x_t, n^y_t)_{t\geq 0}\,$ of the following system of ordinary differential equations (ODEs):
\begin{align}
\frac{dn^x}{dt}= & \Big(r(x)-C(x,x)n^x-C(x,y)n^y+ \displaystyle{\alpha(x,y)\over \beta + \mu\,(n^x+n^y)}\,n^y \Big)n^x = P(n^x,n^y)\,; \nonumber\\
\frac{dn^y}{dt}= & \Big(r(y)-C(y,x)n^x-C(y,y)n^y- \displaystyle{\alpha(x,y)\over \beta + \mu\,(n^x+n^y)}\,n^x \Big)n^y = Q(n^x,n^y) .\label{sysdyn}
\end{align}
\end{prop}

 \me When $\alpha(x,y) \equiv 0$, we get the classical competitive  Lotka-Volterra system. Point $(0,0)$ is an unstable equilibrium and there are 3  stable equilibria: a co-existence equilibrium and two monomorphic equilibria $(\overline n^x,0)$ and  $(0,\overline n^y)$, where
 \be
 \label{equi}\,\overline{n}^{x} = {r(x)\over C(x,x)}\,\ee is the unique stable equilibrium of  the standard logistic equation
\be
\label{logi}\frac{dn}{dt}=  \big(r(x)-C(x,x)n \big)n.\ee
It is well known that the sign  of the invasion  fitness function, defined as
$$f(y;x)= r(y) - C(y,x) \,\overline n^x = r(y) - C(y,x)\,{r(x)\over C(x,x)},$$ governs the stability.
\noindent If $\,f(y;x)<0\,$ and $\,f(x;y)>0$, the system converges to $(\overline n^x,0)$ while  if  $\,f(y;x)>0\,$ and $\,f(x;y)<0$, the system  converges to $(0,\overline n^y)$ and if $\,f(y;x)>0\,$ and $\,f(x;y)>0$, the system  converges to a non trivial co-existence equilibrium.
In the case where the competition kernel  $\,C\,$ is constant and $\,r\,$ is a monotonous function, the fitness function is equal to
  $\,f(y;x)= r(y) - r(x) = - f(x;y)$,  which prevents  co-existence in the limit.

\me When $\alpha(x,y) \neq 0$, the behavior of the system is drastically different as it can be seen in the phase diagrams of Figure \ref{fig:diagram}.

 \me Figure \ref{fig:diagram} shows  eight possible phase diagrams for the dynamical system \eqref{sysdyn}, where the circles and stars indicate stable and unstable fixed points, respectively.   Figures (1)-(4) are possible for all forms of HGT rates, Figures (5)-(6) can happen  in frequency-dependent or  Bedington-deAngelis cases,  while Figures (7)-(8) can be observed only in Bedington-deAngelis case. Compared to the classical two-species Lotka-Volterra system, at least 4 new phase diagrams are possible: Figures (5)-(8).

\begin{figure}
\label{fig:diagram}
\center  \includegraphics[width=12cm]{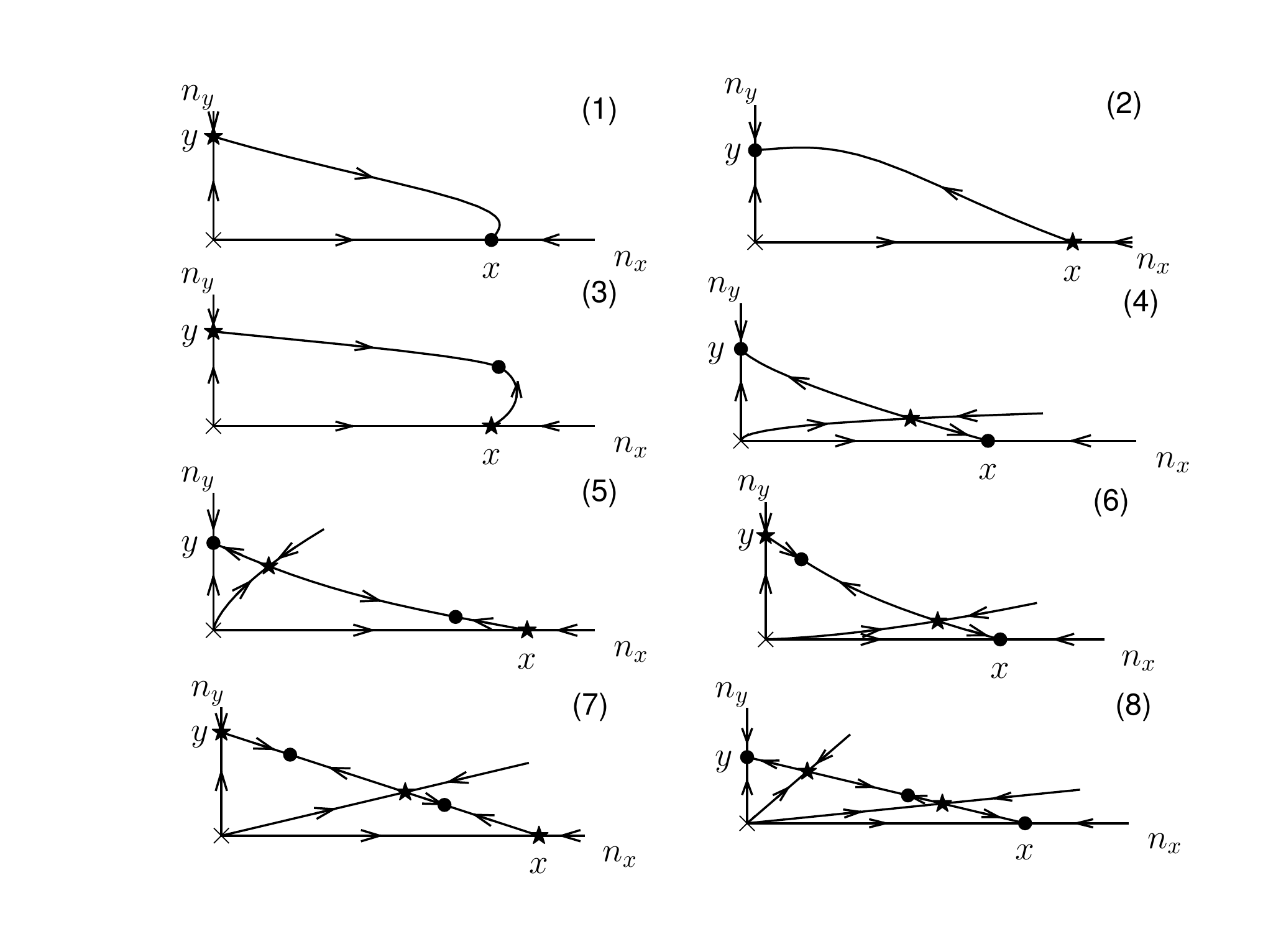}
 \caption{{\small \textit{Phase diagrams for the system \eqref{sysdyn}}}}
\end{figure}

\me Now, define the invasion fitness of individuals with trait $\,y\,$ in the $\,x$-resident population by
\be
\label{fitness} S(y;x) &=&  r(y) - {C(y,x) r(x)\over C(x,x) }+ {\alpha(y,x) r(x)\over \beta C(x,x)+\mu \,r(x)} = f(y;x)+ {\alpha(y,x) r(x)\over \beta C(x,x)+\mu \,r(x)}.\nonumber \\
&&\ee

 \subsection{Properties of the dynamical system \eqref{sysdyn}}
 \label{longte}

\me Let us now analyze the behavior of the system \eqref{sysdyn}.

\bi We first exclude the possibility of cycles contained in the positive quadrant. Recall that a Dulac function $\,\varphi(u,v)\,$ on $(\mathbb{R}_{+}^*)^2$ is a smooth non vanishing  function such that
$$\Big(\partial_{u}(\varphi P) + \partial_{v}(\varphi Q)\Big)(u,v)$$
has the same sign on the domain $(\mathbb{R}_{+}^*)^2$.
\begin{prop}
\label{dulac}
Assume that $\,C(x,x)>0\,$ and $\,C(y,y)>0\,$. Then the function $\,\varphi(u,v) = {1\over u\,v}\,$ is a Dulac function in $(\mathbb{R}_{+}^*)^2$. As a consequence, the system \eqref{sysdyn} has no cycle in $(\mathbb{R}_{+}^*)^2$.
\end{prop}

\me
\begin{proof}
A simple computation gives
$$\partial_{u}(\varphi P) + \partial_{v}(\varphi Q)(u,v) = - \frac{C(x,x)\,u + C(y,y)\,v}{u\,v} < 0,$$
for $\,(u,v) \in (\mathbb{R}_{+}^*)^2$.
The Bendixson-Dulac Theorem (see e.g. \cite[Th.7.12 p.189]{ADL} or \cite[Th.1.8.2, p.44]{GH}) allows to conclude that there is no cycle in the domain.
\end{proof}

\bi From this result and the Poincar\'e-Bendixson theorem (\cite[Section 1.7]{ADL} or \cite[Th.1.8.1, p.44]{GH}) we conclude that any accumulation point of any trajectory starting inside the positive quadrant is either a fixed point or is on the boundary.

\bi
Expressing \eqref{sysdyn} in terms of the size of the population $n_t=n^x_t+n^y_t$ and proportion of trait $x$, $p_t=n^x_t/n_t$, we obtain:
\begin{align}
\frac{dn}{dt}= & n\,\Big(p\,r(x) + (1-p)\,r(y) \nonumber\\
 & \hspace{1cm} - C(x,x)\,p^2n - ( C(x,y)+C(y,x))\,p(1-p)n  - C(y,y)\,(1-p)^2 n\Big) \nonumber\\
\frac{dp}{dt}= & p\,(1-p)\,\Big( r(x) - r(y) \nonumber\\
& \hspace{1cm} + np(C(y,x) - C(x,x)) + n(1-p)(C(y,y)-C(x,y)) + \alpha(x,y){ n\over \beta +\mu n} \Big).\label{eq:edo2}
\end{align}
These equations are generalizations of the classical equations of population genetics with two alleles under selection \cite{Roughgarden71}, in which we have made the influence of demography explicit.  Eq. \eqref{eq:edo2} is useful to investigate the dynamics on the boundary of the positive quadrant which is an invariant set.

\begin{prop}
\label{fixed-points}
Let us recall that $$\overline n^x =  \frac{r(x)}{C(x,x)}\, ;\, \overline n^y =  \frac{r(y)}{C(y,y)}.$$
The points $\, (0, 0)\,$, $\,(0, \overline n^y)\,$ and $\,(\overline n^x, 0)\,$ are the only stationary points on the boundary. The origin is unstable and the two other points are stable for the dynamics on the boundary. Their transverse stability/instability is given by the sign of the fitness function $\,S(x;y)\,$ given in \eqref{fitness}.
\end{prop}

\me The proof is left to the reader. This implies that any accumulation point of any trajectory starting inside the positive quadrant is  a fixed point.
We now investigate the fixed points inside the positive quadrant.

\begin{prop}  Besides the fixed points in the boundary, there is
\begin{enumerate}[i)]
\item
in the BDA case, $\beta \neq 0\,;\,\mu \neq 0$, there are at most $3$     stationary points,

\item
in the FD case ( $\beta = 0\,;\,\mu = 1$), there are at most $2$  stationary points,

\item
in the DD case ($\beta = 1\,;\,\mu = 0$), there is at most $1$  stationary point,
\end{enumerate}
or a line of fixed points inside $\mathbb{R}_{+}^2$.
\end{prop}

\begin{proof}
 It is easier to consider the system in its form \eqref{eq:edo2}.
 The stationary points are denoted by $( n,  p)$ for convenience. They satisfy
 \begin{align*}
0= & n\,\Big(p\,r(x) + (1-p)\,r(y) \nonumber\\
 & \hspace{1cm} - C(x,x)\,p^2n - ( C(x,y)+C(y,x))\,p(1-p)n  - C(y,y)\,(1-p)^2 n\Big) \nonumber\\
0= & p\,(1-p)\,\Big( r(x) - r(y) \nonumber\\
& \hspace{1cm} + np(C(y,x) - C(x,x)) + n(1-p)(C(y,y)-C(x,y)) + \alpha(x,y){ n\over \beta +\mu n} \Big).
\end{align*}
   If $n\neq 0$ and $p\notin \{0,1\}$,  we deduce from the first equation that
  $$n = \frac{p r(x)+(1-p)r(y)}{Q( p)}$$where$$Q( p)=C(x,x) p^2 + (C(x,y) + C(y,x))p(1-p) +C(y,y) (1-p)^2\neq 0$$
 for $p\in(0,1)$.
Replacing $n$ by this quantity, we write   the second equation as
\begin{multline*}
0=\frac{p(1-p)}{Q( p)\big(\beta Q( p) + \mu(p r(x) +(1-p)r(y))\big)} \times \\
\begin{aligned}
&\bigg((r(x)-r(y))Q( p)\big(\beta Q( p) + \mu(p r(x) +(1-p)r(y)\big) \\
&\qquad +(p r(x) +(1-p)r(y)) \big(\beta Q( p) + \mu(p r(x) +(1-p)r(y))\big)\\
&\qquad \qquad \qquad \qquad
 \big(p(C(y,x)-C(x,x))+(1-p)(C(y,y)-C(x,y)\big)\\
&\qquad + \alpha(x,y) \big(p r(x) +(1-p)r(y)\big)Q( p)\bigg) .
\end{aligned}
\end{multline*}
When $\beta\neq 0$ and $\mu\neq 0$ (BDA case),   the term between the large brackets   is  a priori a polynomial in $p$ of degree $4$. But explicit computation shows that the term of order $4$ vanishes. Then this polynomial is of degree $3$ and there are at most $3$ stationary points inside the domain. In  FD cases, the expression simplifies as $\frac{p(1-p)}{Q( p)}$ times a polynomial of degree $2$ and there are at most two stationary points. The DD case reduces to a Lotka-Volterra system. \end{proof}

\noindent To obtain more insight on the limiting dynamics, we use the Poincar\'e index (see \cite[Chapter 6]{ADL} or \cite[p.50-51]{GH}).\\

\smallskip \noindent Let us first remark  that  the trace of the Jacobian matrix of  any  fixed point $(u_{0}, v_{0})$ inside $\mathbb{R}_{+}^2$, is equal to
$$- C(x,x) \,u_{0} - C(y,y) \,v_{0} <0.$$ This implies that any fixed point inside the positive quadrant is either a sink (index $1$), a saddle (index $-1$) or a non-hyperbolic point of index $0$ with a negative eigenvalue of the Jacobian matrix (because the vector field is analytic, see \cite[Th.6.34]{ADL}).
We use the circuit  with anticlockwise orientation drawn in Fig. \ref{fig:2}.   The largest radius is chosen large enough such that there are no  fixed points outside the loop. The fixed points $ (\overline n^{x},0)$ and $ (0,\overline n^{y})$ on the boundaries are denoted by $A$ and $a$ on Fig. \ref{fig:2}.

\begin{figure}[!ht]
\label{fig:2}
\center \includegraphics[width=6cm]{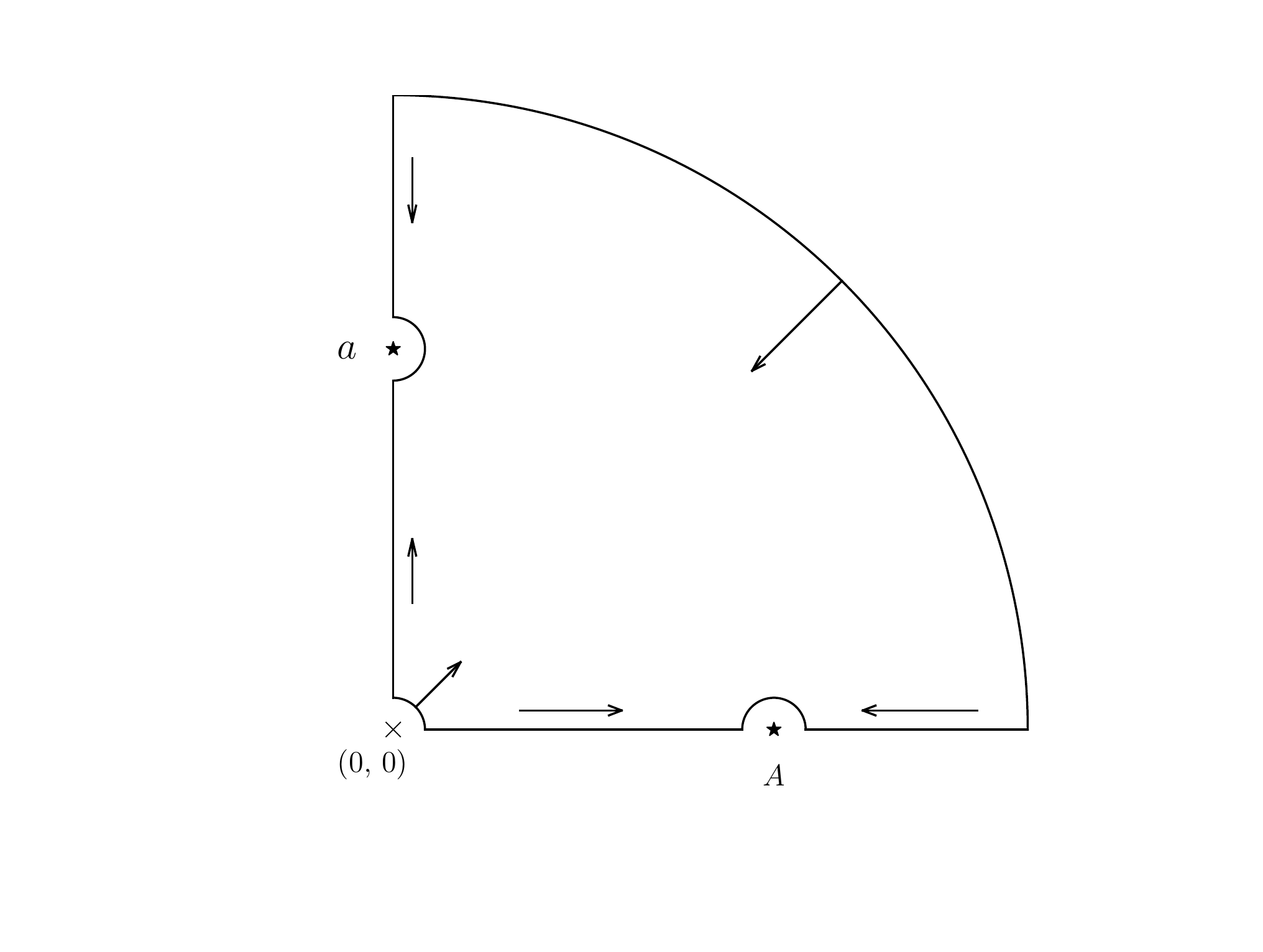}
\caption{{\small\textit{Circuit used to compute the Poincar\'e index and determine the nature of fixed points inside the positive quadrant.}}}
\end{figure}
\noindent The arrows represent the directions of the vector field along the different parts of the circuit.  It can be shown in all cases that for a radius large enough, the large arc contributes $1/4$ to the index.

\begin{prop}
\label{fp}
Assume all  fixed points are hyperbolic. The only possibilities are as follows:

- if $(\overline n^{x},0)$ and $ (0,\overline n^{y})$ are unstable points, the index of the circuit is $1$ and there is either one stable point inside the domain or
$3$ fixed points: $2$ stable nodes and one saddle point.

- if $(\overline n^{x},0)$ and $ (0,\overline n^{y})$ are stable points, the index is -1 and there is  either one saddle point inside  or
$3$ fixed points: $2$ saddle points and one stable point.

- if one of the points  $(\overline n^{x},0)$ or $ (0,\overline n^{y})$ is an unstable node and the other one a saddle point, then the index is $0$ and we have either $0$ fixed point or two fixed points: one saddle point and one stable point.
\end{prop}

\me This proposition follows from the Poincar\'e-Hopf theorem: the index of the curve is equal to the sum of the indices of the fixed points inside the domain (see \cite[Prop.6.26, p.175]{ADL} or \cite[Prop.1.8.4, p.51]{GH}). Combining this result with  Proposition \ref{fixed-points}, one can decide between the different possibilities depending on the parameters.

\me The diagrams in Figure \ref{fig:diagram}  realize the different situations described above. However, there may exist other diagrams in accordance with Proposition \ref{fp} that we have never observed numerically. We are yet unable to prove or disprove the existence of such other diagrams. One can nevertheless show that in the case where $x$ and $y$ are sufficiently similar, the phase diagrams of Figure \ref{fig:diagram} are the only possible ones (cf. \cite{billiardcolletferrieremeleardtran}).
In the case of non hyperbolic fixed points inside the positive quadrant (with index $0$ as mentioned previously), an analogue of Proposition \ref{fp} can be established. This situation is however exceptional since it implies a nonlinear (polynomial) relation between the coefficients.

\subsection{The case of constant  competition}
\label{sec:Cconst}

\me Assume that the competition kernel is constant  $C(u,v)\equiv C$ for all $u,v\in {\cal X}$.
Eq.\eqref{eq:edo2}  gives:
\begin{align}\begin{cases}\label{eq:edo2-Cconstant}
\frac{dn}{dt}= & n\,\big(p\,r(y) + (1-p)\,r(x)  - Cn\big) \\
\frac{dp}{dt}= & p\,(1-p)\,\Big( r(y) - r(x)  + \alpha(y,x){ n\over \beta +\mu n} \Big).
\end{cases}\end{align}

\me
\noindent Let us consider separately the cases of  FD transfer rate  and  DD or BDA transfer rates.

\me
\emph{Frequency-dependent horizontal transfer rate.} With $\,\beta=0\,$ and $\,\mu=1\,$, \eqref{eq:edo2-Cconstant}
shows that there are only two equilibria for the second equation: $\,p=0\,$ or $\,p=1\,$ (Figures \ref{fig:diagram} (1)-(2)). Therefore there is no polymorphic fixed point and we get a very simple ``Invasion-implies-Fixation'' criterion:
trait $y$ will invade a resident population of trait $x$ and get fixed if and only if
\be\label{eq:FDpolym}  S(y;x)&=& f(y;x) + \alpha(y,x) = - S(x;y)>0.
\ee
Thus, compared to a system without HGT, horizontal transfer can revert the direction of selection (\emph{i.e.} $S(y;x)$ and $f(y;x)$ have opposite signs) provided that
$$|\alpha(y,x)| > |f(y;x)|\quad \mbox{ and } \quad \Sgn(\alpha(y,x)) = -\Sgn(f(y;x)).$$
This implies that HGT can drive a deleterious allele to fixation.

\bi \emph{ Density-dependent or BDA horizontal transfer rate.} When $\beta\not=0$, there exists a polymorphic fixed point when
\be
\label{eq:DDpolym}
0<\widehat{p}=-\frac{f(y;x)(\beta C+ \mu r(x)) +\alpha(y,x) r(x) }{\mu f(y;x)^2+ \alpha(y,x)\,f(y;x)}<1.
\ee
If $f(y;x)$ and $\alpha(y,x)$ are both positive, the above expression is negative and there is fixation of $y$. If $f(y;x)$ and $\alpha(y,x)$ are both negative, $\widehat{p}<1 \Longleftrightarrow -f(y;x) \beta C<r(y)(\mu f(y;x)+\alpha(y,x))$ which never happens since the left hand side is positive and the right hand side is negative. So there is fixation of $x$ in this case.
When $f(y;x)$ and $\alpha(y,x)$ have opposite signs, there may exist a non-trivial fixed point
which is stable if
\be
\label{trade}\mu f(y;x)+\alpha(y,x)>0.\ee

\me In contrast to the classical Lotka-Volterra competition model in which constant competition prevents stable coexistence, HGT with DD or BDA transfer rates allows the maintenance of a deleterious trait ($f(y;x) < 0$) in a stable polymorphic state; this requires that the flux rate $\,\alpha(y,x)\,$ be positive and large enough in favor of $y$ to $x$.

 \section{Rare mutation probability in the evolutionary time-scale}
\label{evolution}

As  seen in Section 3, it is not possible to capture the effect of rare mutations ($p_K\rightarrow 0$) at the ecological time scale. We have to consider  a much longer time scale to observe  this effect. The mutation time scale is of order  ${1\over K \,p_K}$ and we  will assume in the following that  when $K$ is large enough,
 \begin{equation}
\forall\, V>0,\quad  \log K\, \ll\, {1\over K\,p_K} \,\ll\, e^{VK}.
\label{pk:TSS}
\end{equation}
A separation of time scales between competition phases and mutation arrivals results from this assumption.
 Indeed, mutations being   rare enough, the  selection will have  time to eliminate deleterious traits or to fix advantageous  traits before the arrival of a new mutant.

 \me Let us now give a rigorous approach of the mechanism governing the successive invasions of successful mutants.

 \subsection{Probability and time of invasion and fixation under competition with horizontal transfer}\label{invasion}

As an intermediate step,  we investigate the fate of a newly mutated individual with trait $y$ in a resident population
 in which trait $x$ is common.  We assume that the invasion fitness of trait $y$ defined in \eqref{fitness} is positive, $S(y;x)>0$. This includes both cases of an advantageous trait ($f(y;x)>0$), or a deleterious trait ($f(y;x)<0$) provided that the horizontal transfer  rate from $y$ to $x$ is high enough. Figure \ref{fig:fix} gives illustration of the determining effect of the transfer in the evolution. It  shows the different stochastic dynamics one can obtain under frequency or density-dependent HGT in the simple case of unilateral transfer and that could not be realized without transfer. Figure \ref{fig:fix} shows that a deleterious trait can invade a resident population and go to fixation (Fig.\ref{fig:fix}(b) and (c)), or can stably coexist with the resident one (Fig.\ref{fig:fix}(a) and (d)). Fig.\ref{fig:fix}(a) especially shows that both traits stably coexist even though competition is constant, which is made possible by density-dependent HGT (we have recalled in Subsection \ref{sec:LV} that it cannot occur for a usual Lotka-Volterra system).

 \begin{figure}[!ht]
 \label{fig:fix}
 \begin{center}
\begin{tabular}{cc}
\includegraphics[height=5cm,trim =15mm 60mm 25mm 80mm, clip, scale=0.5]{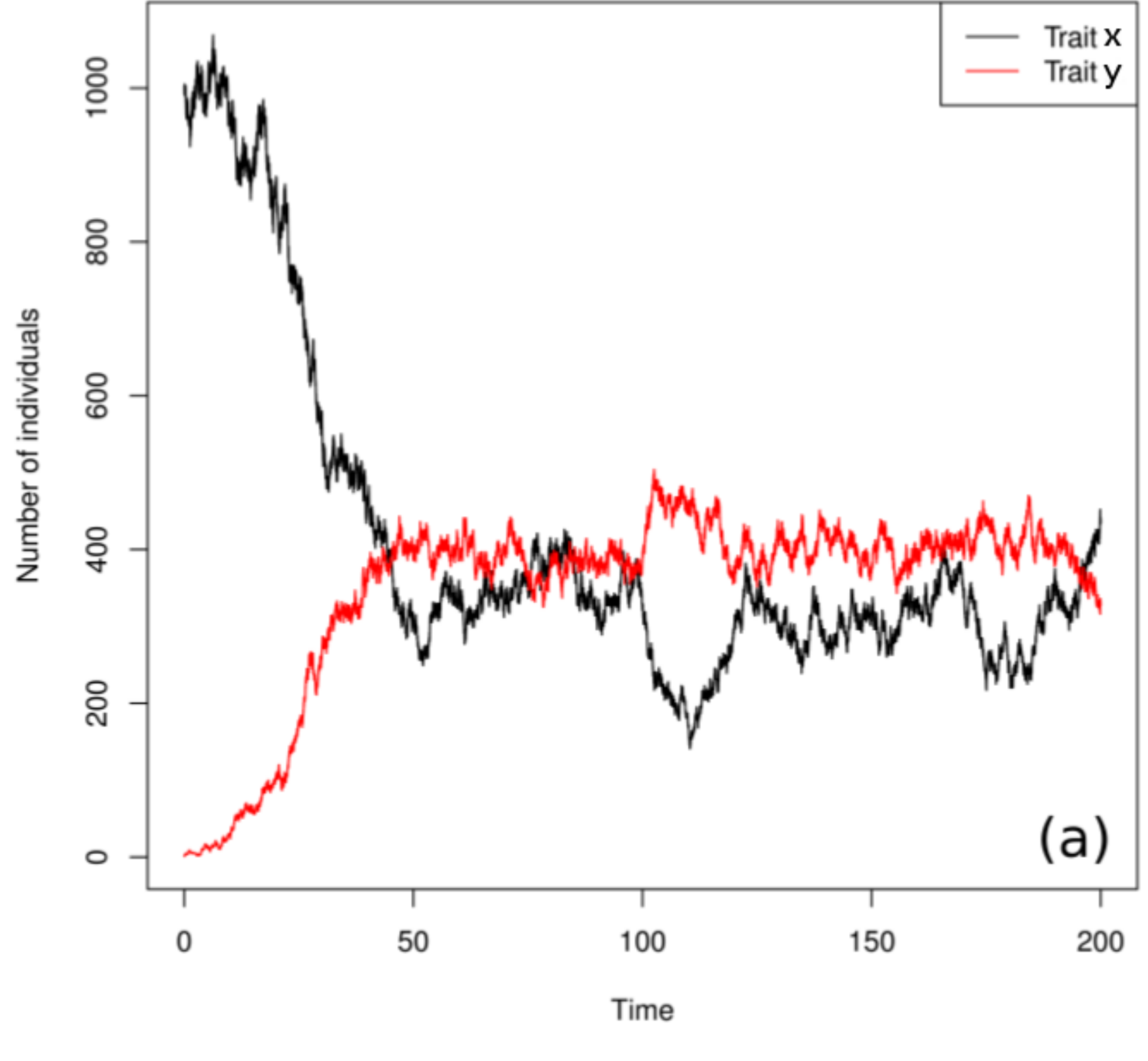} &
\includegraphics[height=5cm,trim =15mm 60mm 25mm 80mm, clip, scale=0.5]{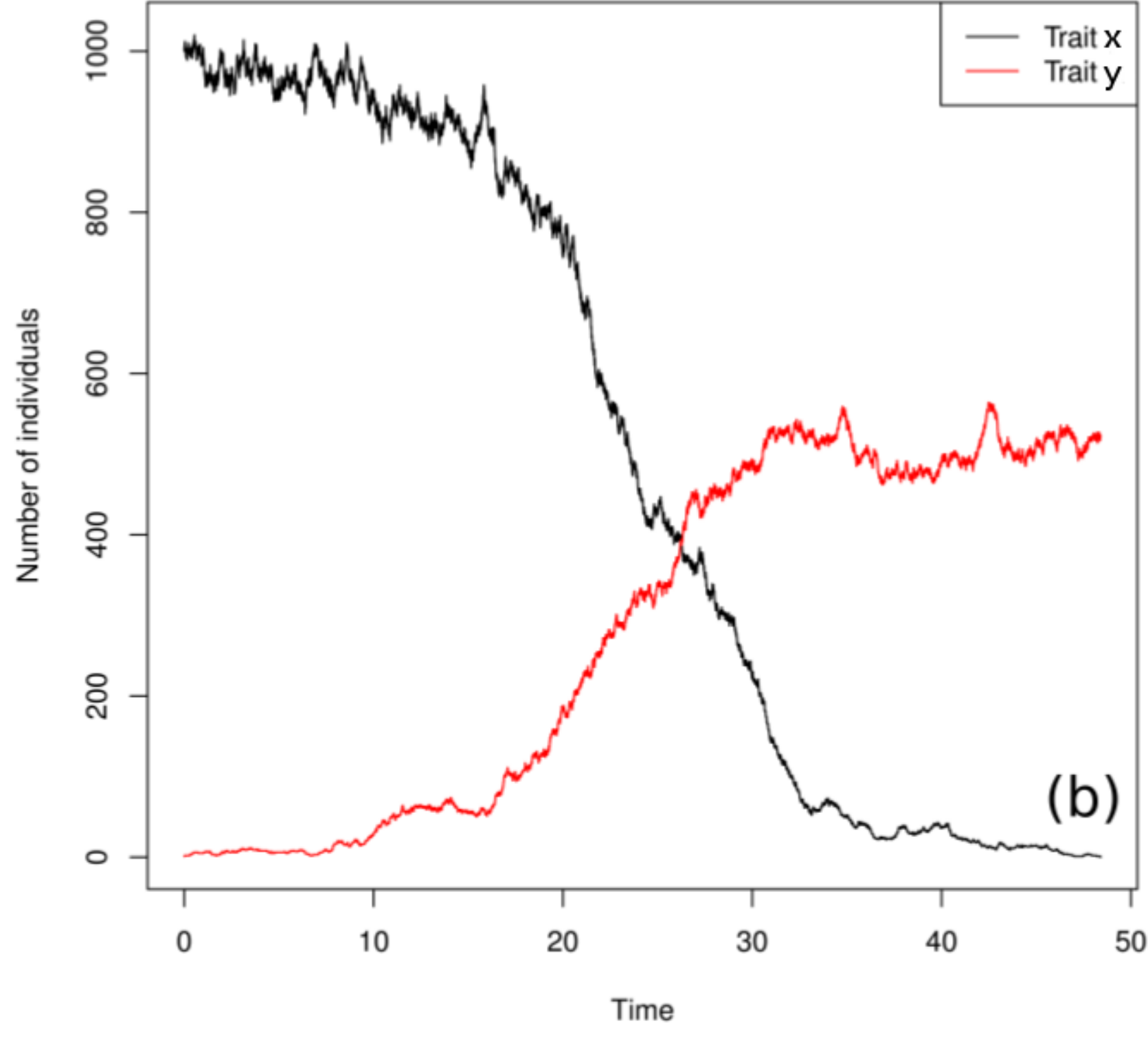} \\
\includegraphics[height=5cm,trim =15mm 60mm 25mm 80mm, clip, scale=0.5]{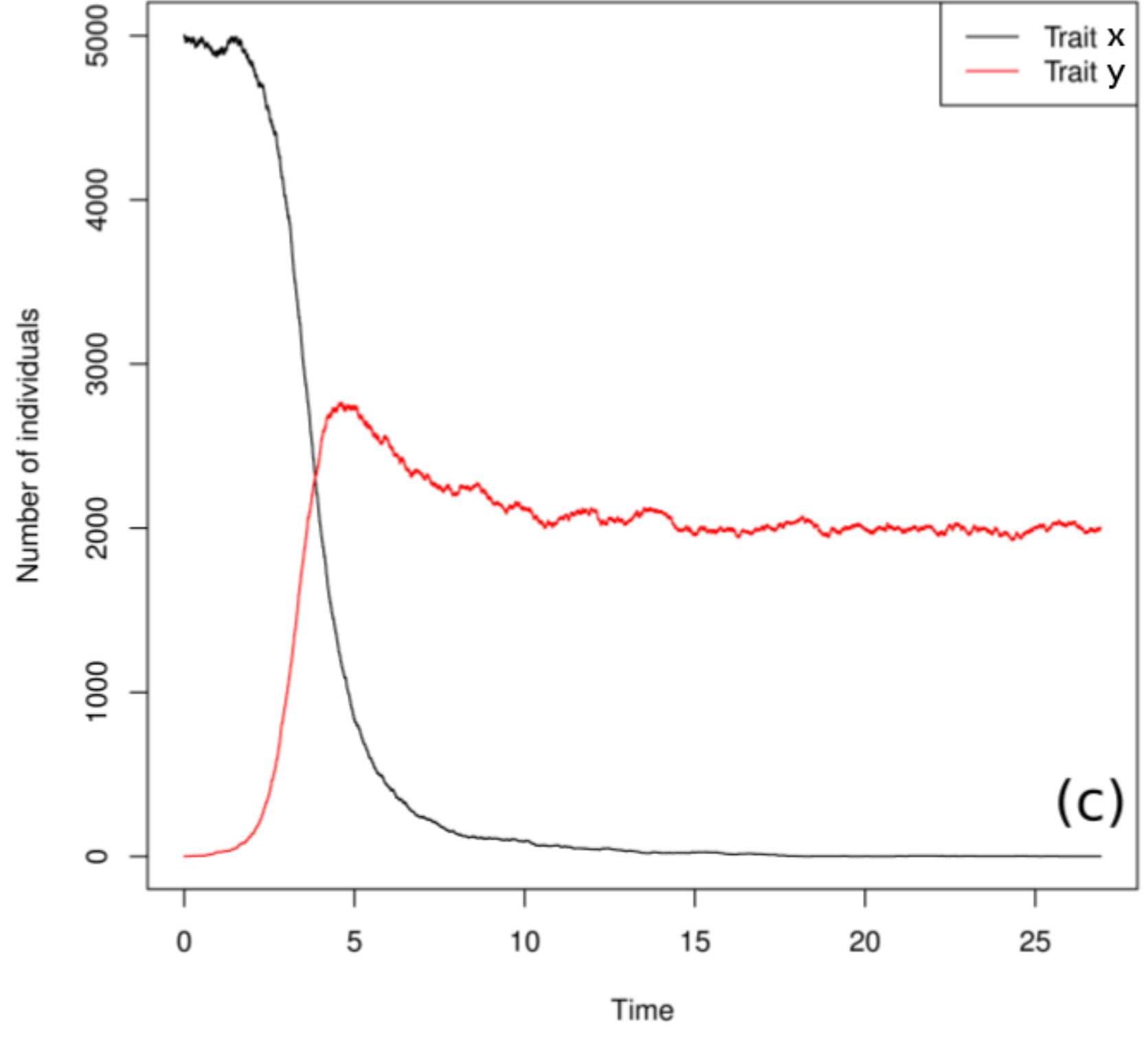} &
\includegraphics[height=5cm,trim =15mm 60mm 25mm 80mm, clip, scale=0.5]{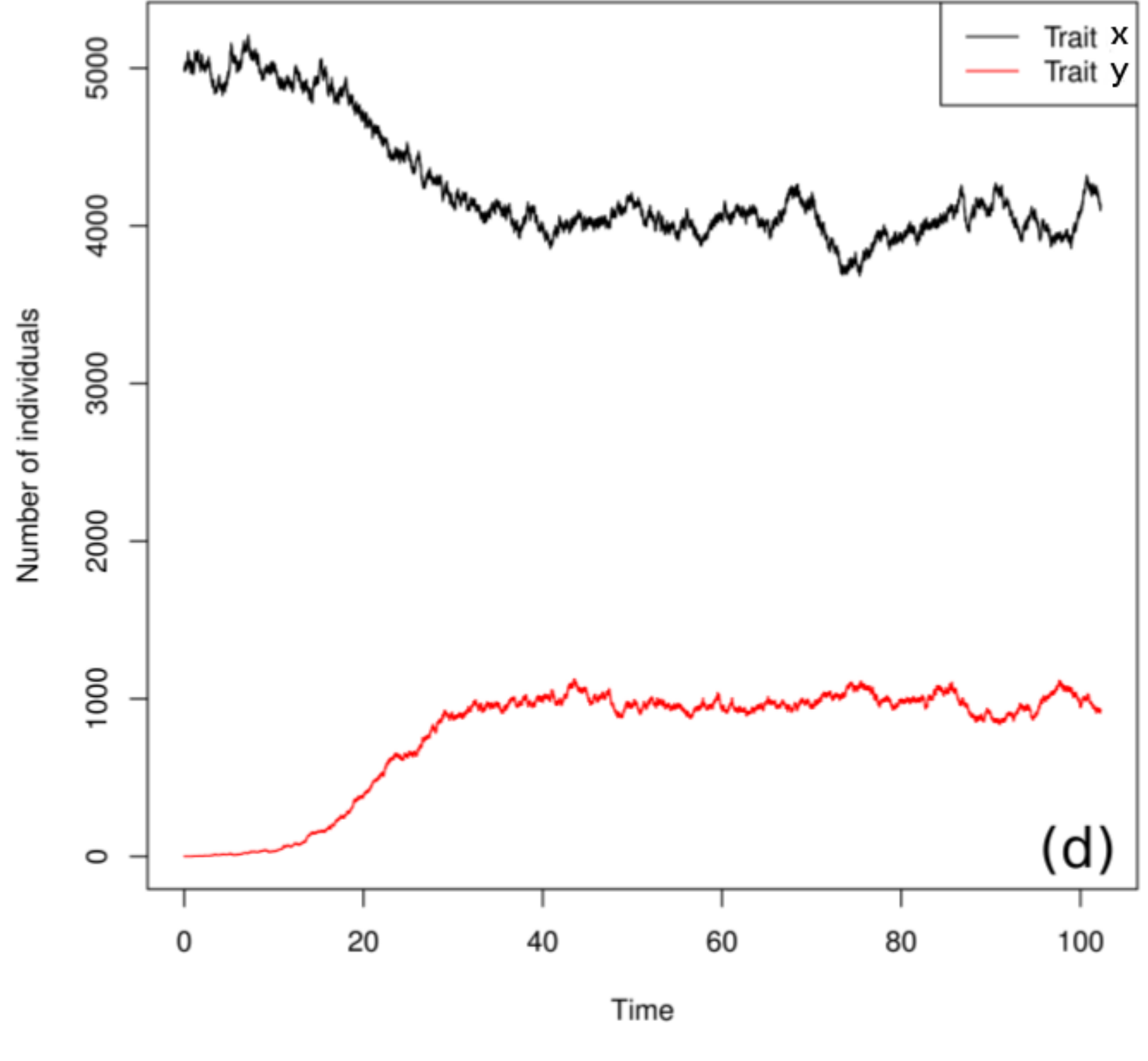} \\
\end{tabular}
\end{center}
\caption{{\small \textit{Invasion and fixation or polymorphic persistence of a deleterious mutation with density-dependent (left, (a) and (c), $\mu=0$, $\beta=1$) or frequency-dependent (right, (b) and (d), $\mu=1$, $\beta=0$,) unilateral HGT rates. The deleterious nature of the mutation means that its invasion fitness without HGT is negative. Other parameters: Top figures (a) and (b): constant competition coefficients $C(y,x)=C(x,y)=C(x,x)=C(y,y)=1$, $b(y)=0.5$, $b(x)=1$, $d(x)=d(y)=0$ $K=1000$, $\alpha =0.7$; Bottom figures (c) and (d): $C(y,x)=C(x,x)=2$, $C(y,y)=4$, $C(x,y)=1$, $b(y)=0.8$, $b(x)=1$,  $d(x)=d(y)=0$, $K=10000$,  $\alpha =5$ under density-dependent rate, $\alpha=0.5$ under frequency-dependent rate.}}}
\end{figure}

\me Consider an individual with trait $y$ introduced in a resident population of individuals with trait $x$, whose size $N^{x,K}$ is very close to the equilibrium  $K \overline n^x$.  During the first phase, $N^{y,K}$ is very small with respect to $N^{x,K}$. It  can be approximated by a linear birth and death branching stochastic process, at least until it reaches the threshold $\eta\, K$, for a given $\eta>0$.
In this birth and death process, the transfer $x \to y$ acts as a birth term and the transfer $y\to x$ as a death term.
When $K$ tends to infinity, the  probability for the process  $N^{y,K}$ to reach $\eta K$ is approximatively  the survival probability of the process  (e.g. \cite{champagnat06,champagnatferrieremeleard}) and is given by
\begin{equation}
P(y;x)=\frac{S(y;x)}{b(y)+ h(y,x,\overline n^{x} \delta_x)\,\overline n^{x}}= \displaystyle{\frac{b(y)- d(y) + \big( \frac{\alpha(y,x)}{\beta+\mu \,\overline n^{x}} -  C(y,x)\big)
\,\overline n^{x}}{b(y)+ \frac{\tau(y,x) \,\overline n^{x}}{\beta+\mu\, \overline n^{x}}}}.\label{eq:proba_invasion}
\end{equation}



\me
\subsection{Times of invasion and fixation}\label{fixation}

\me We refer here to \cite{champagnat06}, where  the results are rigorously proved.
As the selectively advantageous trait $y$ increases from rare, the first phase of the $y$-population growth has a duration of order $\log K/S(y;x)$. If $N^{y,K}$ reaches the threshold $\eta K$, then the second phase begins, where the processes $(N^{x,K},N^{y,K})$ stay close to the dynamical system \eqref{sysdyn}.  The deterministic trajectory, which has a duration of order 1,  can reach one of two final states: either both types of individuals stably coexist, or individuals with trait $y$ invade the population and the $x$-population density reaches the threshold $\eta$ (i.e. $N^{x,K}_t< \eta K$). Should the latter happens, the third phase begins and $N^{x,K}$ can be approximated by a subcritical linear birth and death branching process, until $y$ is fixed and $x$ is lost.
In this birth and death process, the transfer $y \to x$ acts as a birth term and the transfer $x\to y$ as a death term.
The third phase has an expected duration $\E_{\eta K} \left[ T_0 \right]$ given by   (see \cite[Section 5.5.3, p.190]{Meleardbook})
\ben
\E_{\eta \,K} \left[ T_0 \right]=\frac{1}{b} \sum_{j\geq 1} \Big( \frac{b}{d}\Big)^j \ \sum_{k=1}^{\eta K -1}\frac{1}{k+j},
\een
where
$\,
b=  \ b(x)+\frac{\tau(x,y)r(y)}{\beta C(y,y)+\mu r(y)}\ ,\ d=  \ d(x)+\frac{C(x,y)r(y)}{C(y,y)}+\frac{\tau(y,x)r(y)}{\beta C(y,y)+\mu r(y)}.
$

\me When $K \rightarrow \infty$, $\,\E_{\eta K} \left[ T_0 \right] \simeq \log K/(d-b)$, which means that the third phase is of order $\log K/|S(x;y)|$ in duration. Summing up, the fixation time of an initially rare trait $y$ going to fixation is of order
\be \label{eq:Tfix}
T_{fix}=\log K \Big({1 \over S(y;x)}+ {1\over {|S(x;y)}|}\Big) +O(1), 
\ee
where the expressions for $S(y;x)$ and $S(x;y)$ are given in \eqref{fitness} and $O(1)$ is a negligible  term.\\

\subsection{The Trait Substitution Sequence}\label{Sec:TSS}

 The   limiting     population process at the mutation time scale ${t\over K p_K}$ will describe the evolutionary dynamics of invasions of successful mutants.

\me Let us assume in what follows that  the ecological  coefficients impede the  coexistence of two
 traits. This is known as the {\it Invasion Implies fixation} (IIF) assumption.

\me {\bf  Assumption \emph{(IIF)}}:
 \emph{Given any $x\in {\cal X}$ and  Lebesgue almost any $y\in {\cal X}$:   either  $(\overline n^{x},0)$ is a stable steady state
of \eqref{sysdyn}, or we have that $(\overline n^{x},0)$ and  $(0,\overline n^{y})$ are respectively unstable and stable steady states, and that any solution of \eqref{sysdyn} with initial state in $(\mathbb{R}_{+}^*)^2$ converges
to $(0,\overline n^{y})$ when $t\to \infty$.
}

 \bi From Section \ref{section:ODE2} we know that invasion does not necessarily imply fixation, even when the invasion fitnesses of the two types have opposite signs, as shown by Fig. \ref{fig:diagram} (5) and (6). In these cases, fixation depends on initial conditions and is usually not achieved when the invading type starts from a small density. Considering the special case of constant competition, however, invasion does imply fixation (cf. Section \ref{sec:Cconst}) if HGT rates are FD or when condition \eqref{eq:DDpolym}
 is not satisfied if HGT rates are DD or BDA.

 \bi Assumption \eqref{pk:TSS} together with Assumption \emph{(IIF)} imply that for a monomorphic ancestral
population, the dynamics at the time scale $t/(Kp_K)$ can be approximated by  a jump process over  singleton measures on ${\cal X}$ whose mass at any time is at equilibrium. More precisely, we have

\begin{thm}
\label{TSS}
We work under Assumptions $(H)$, \eqref{pk:TSS} and (IIF).
The initial conditions are
   $\nu^K_{0}(dx)= N^K_{0}\,\delta_{x_{0}}(dx)$ with $x_0\in  {\cal X}$, $\ \lim_{K\to \infty} N^K_{0} = \overline n^{x_{0}}$ and $\,\sup_{K\in \N^*} \E((N^K_0)^3)<+\infty$.

\me
Then, the sequence of processes $\ (\nu^K_{./(Kp_K)})_{K\ge 1}$ converges in law to the $M_{F}( {\cal X})$-valued process  $\,(V_t(dx)=  \overline{n}^{Y_{t}}\,\delta_{Y_{t}}(dx), t\geq 0)$
where the process $(Y_{t})_{t\geq 0}$ is a pure jump process on $ {\cal X}$, started at $x_0$, which  jumps from $x$ to $y$ with the jump measure
  \begin{equation}
    \label{taux}
    b(x) \,\overline n^{x}\,{[P(y;x)]_+}
    \,m(x,dy),
  \end{equation}
  where $P(y;x)$ has been defined in \eqref{eq:proba_invasion}.

  \noindent
The convergence holds in the sense of finite dimensional distributions on $M_{F}( {\cal X})$ and in the sense of occupation measures in $M_F( {\cal X} \times [0,T])$ for every $T>0$.\hfill $\Box$
\end{thm}

\me
 The jump process
  $\ (Y_t,t\geq 0)$ (with $Y_0=x_0$)
describes the support of $(V_t, t\geq 0)$.  It  has been heuristically introduced in    \cite{metzgeritzmeszenajacobsheerwaarden} and  rigorously studied in \cite{champagnat06}, in the case without transfer. It is often called Trait Substitution Sequence (TSS).

\bi
\begin{rem} Let us remark that the transfer events can change the direction of evolution.
 For example, let us consider the size model with trait $x\in[0,4]$, where $\beta=0$, $\mu = 1$, $C$ is constant and
\be \label{exple} b(x) = 4-x\,,\, d(x) = 0\,,\, \tau(x,y) =e^{x-y}.\ee
Then if $h>0$,
$$S(x+h;x) = r(x+h) - r(x) + \tau(x+h,x) -\tau(x,x+h)  = - h + e^h - e^{-h},$$
which is positive if and only if $h>0$.
Thus the evolution with transfer is directed towards larger and larger traits.

\me On the other hand, without transfer, the invasion fitness $f(x+h;x)$  is negative for $h>0$ and a mutant of trait $x+h$ with $h>0$ would not appear in the TSS asymptotics. Therefore, adding the transfer drastically changes the situation.
\end{rem}

\bi \begin{proof}[Proof of Theorem \ref{TSS}]The proof is a direct adaptation of  \cite{champagnat06}. Accounting for the transfer parameters, the birth and death rates, respectively  of the resident $x$ and mutant $y$, become
$$b(x)+{\tau(x,y) N^{y,K}\over  \beta + \mu N^K}\quad,\quad d(x)+C(x,x) N^{x,K}+ C(x,y) N^{y,K}+{\tau(y,x) N^{y,K}\over \beta + \mu N^K}\,;$$
 $$b(y)+{\tau(y,x) N^{x,K}\over \beta + \mu N^K}\quad , \quad d(y)+C(y,x) N^{x,K} + C(y,y) N^{y,K}+{\tau(x,y) N^{x,K}\over  \beta + \mu N^K}.$$

\me The main idea is as follows. If mutations are rare, the selection has time to eliminate the deleterious traits or to fix
advantageous traits before a new mutant arrives. We can then combine the results obtained in Sections \ref{section:ODE2}, \ref{invasion} and \ref{fixation}.

\noindent  Let us fix $\eta>0$. At $t=0$, the population is monomorphic with trait $x_{0}$ and satisfies the assumptions of Theorem \ref{TSS}. As long as no mutation occurs, the population stays monomorphic with trait $x_{0}$ and for $t$ and $K$ large enough, the density process $\,\langle\nu^K_t, {\bf 1}_{x_{0}}\rangle\,$
belongs to the $\eta$-neighborhood of $\overline{n}^{x_{0}}$ with large probability (cf. Prop. \ref{largepopulationdiploid}).  We know by the large deviations  principle (see for example  Freidlin-Wentzell \cite{FW84} or Feng-Kurtz~\cite{FK}) that the time taken by the density process to leave the
$\eta$-neighborhood of $\overline{n}^{x_{0}}$ is larger than $\exp(VK)$, for some $V>0$, with high probability. Hence if Assumption  \eqref{pk:TSS} is satisfied, the first mutant will
appear with large probability before the population process exits the $\eta$--neighborhood of $\overline{n}^{x_{0}}$. Therefore,  the approximation of the population process by $\overline{n}^{x_{0}} \delta_{x_{0}}$ stays valid until the first mutation occurence.

\noindent  The invasion dynamics of the mutant with trait $y$ in this resident population has been  explained  in Subsections  \ref{invasion} and \ref{fixation}.   At the beginning,  since the density of the resident population is
close to $\overline{n}^{x_{0}}$, the mutant dynamics is close to a linear birth and death process whose rates depend on $\overline n^{x_{0}}$.
If  $S(y;x_{0})>0$, the birth and death process is supercritical, and therefore, for large $K$, the probability that the mutant population's density attains $\eta$ is close to the survival probability $P(y;x_{0})$.
After this threshold, the density process $(\langle\nu^K_t, {\bf 1}_{x_{0}}\rangle, \langle\nu^K_t, {\bf 1}_{y}\rangle )$ can be approximated, when $K$ tends to infinity, by  the
solution of the dynamical system   \eqref{sysdyn}. It will attain with large probability an
$\eta$-neighborhood of the unique globally asymptotically stable equilibrium $\,n^*\,$ of~\eqref{sysdyn}  and will stabilize  around this equilibrium. Using Assumption \emph{(IIF)},  we know that the equilibrium is  $(0,\overline{n}^{y})$.
It is also shown in Subsection \ref{fixation} that if the initial
population is of order $K$,  then the time between the occurence of the mutant and the final stabilization  is given by \eqref{eq:Tfix}. Hence, if $\ \log K \ll {1\over K p_K}$, with a large probability
 this phase of
competition-stabilization will happen  before the occurrence of the next mutation. Using Markovian arguments we can reiterate the reasoning after every mutation event. If the process belongs to a $\eta$-neighborhood of $\overline{n}^x$, the mutation rate from an individual with
trait $x$ is close to $ p_K  b(x) K \overline{n}^x$. Hence, at the time scale ${t\over K p_K}$, it is approximatively $b(x) \overline{n}^x$.
 Therefore, if we consider the population process at time $t/K p_{K}$ and make $K$ tend to infinity, we only keep in the limit the successive stationary states corresponding to
successive advantageous  mutations. The limiting process is a pure jump process $(V_t, t\geq 0)$ such that the jump rate from a state  $
\overline{n}^x \delta_{x}$ to a state $\overline{n}^y \delta_{y}$ is $b(x) \overline{n}^x\,[P(y;x)]_+$. The mutant trait $y$ is chosen following $m(x,dy)$.  This explains Formula~\eqref{taux}.
\end{proof}

\section{Canonical equation of the adaptive dynamics }
\label{sec:7}

\me The impact of transfer on evolution can also be captured and highlighted  with the canonical equation. The canonical equation, first introduced by Dieckmann Law \cite{dieckmannlaw} (see also \cite{champagnatferrieremeleard}) is the limit of the TSS when we accelerate further time and consider small mutation steps.

\me
Let us now assume that the mutations are very small in the sense that the mutation distribution $m_{\varepsilon}$ depends on a parameter $\varepsilon>0$ as follows:
$$\int g(z) m_{\varepsilon}(x,dz) =  \int g(x+\varepsilon h) m(x,dh),$$
where $m$ is a reference symmetric measure.
Then the generator of the TSS  $Y^\varepsilon$ (which now depends on the parameter $\varepsilon$),  is given by
$$L^\varepsilon g(x) = \int (g(x+\varepsilon h) - g(x))\,  b(x)\, \overline n^{x}\,\frac{[S(x+\varepsilon h;x)]_+}
    {b(x+\varepsilon h)+\tau(x+\varepsilon h,x)\overline{n}^x}\,m(x,dh).$$

\me    If we assume that $x \mapsto \tau(x,y)$ and $x\mapsto b(x)$ are continuous and since $f(x;x)=\tau(x,x)=0$, then when $\varepsilon$ tends to $0$,  $L^\varepsilon g(x) $ converges to
$$ {1\over 2} g'(x) \, \overline n^{x}\,\partial_{1}S(x;x)
  \int h^2 \,m(x,dh).$$
Then standard tightness and identification arguments allow us to show the convergence in law in the Skorohod space   $\mathbb{D}([0,T], {\cal X})$ of the process ${1\over \varepsilon^2} Y^\varepsilon$ to the deterministic equation
\be
\label{eq:canonique}
x'(t) &=& {1\over 2}\,   \bar n^{x(t)}\,\partial_{1}S(x(t);x(t)) \int h^2\, m(x(t),dh),
\ee
 the so-called
{\it canonical equation of adaptive dynamics} introduced in \cite{dieckmannlaw}.

\me When the mutation law  $m$ is not symmetric, \eqref{eq:canonique} involves  the whole measure $m$, instead of  its variance.

\bi Let us come back to the example \eqref{exple} introduced previously. In this case, the canonical equation is given by
$$x'(t)  = \, {4-x(t)\over C}   \int h^2 \,m(x(t),dh),$$
since $r'(x) = -1$ and $\partial_{1}\tau(x,x) = - \partial_{2}\tau(x,x) = 1$.
Then the trait support is an increasing function. That means that the evolution  with transfer decreases  the reproduction rate until it vanishes,  and therefore  drives the population to an evolutionary suicide.
Let us remark that without transfer, the canonical equation would be
$$x'(t)  = \,- {4-x(t)\over C}   \int h^2\, m(x(t),dh),$$
and would drive to the optimal nul trait which maximizes the birth rate.

\bi Then we observe that transfer can drastically change the direction of evolution, leading in the worst cases to an evolutionary suicide.
Such situation will be observed on the numerical simulations of the next section.

\section{Simulations - Case of Frequency-Dependence}

 \emph{(With the help of the master students Lucie Desfontaines and St\'ephane Krystal)}.

 \bi
 In this section, we focus on the special case of unilateral transfer, which is relevant to address the question of fixation of mobile genetic elements such as plasmids. Plasmid transfer is unilateral: individuals containing a specific plasmid can transmit one copy to another individual which does not carry this plasmid. Let us assume that trait $y$ indicates that the individual carries the plasmid of interest; individuals with trait $x$ are devoid of this plasmid. Unilateral transfer then means $\tau(y,x)>0$ and $\tau(x,y)=0$, hence $\alpha(y,x)=\tau(y,x)$.


\bi
 The next simulations will be concerned with the particular case of  frequency-dependent  unilateral HGT model  with $x\in [0,4]$, $\,m(x,h)dh ={\cal N}(0,\sigma^2)$,  $\displaystyle{\tau(x, y, \nu) = {\tau\,{\bf 1}_{x>y}\over \langle \nu,1\rangle}}$.

\me   $b(x) = 4-x\ ;\ d(x) = 1\ ;\ C=0,5\ ;\ p=0,03\ ; \ \sigma=0,1\ ;\ K=1000$.

\me  Initial state: $\, 1000\,$ individuals with trait $\,1$. Equilibrium of population size with trait $\,1$:  $1000 \times {b(1)-d(1)\over C} = 4000$ individuals.

\me

\me The constant $\,\tau\,$ will be the varying parameter. In the rest of the section we present different simulations highlighting the influence of $\tau$ and show how, depending on $\tau$, we can obtain drastically different behaviors, from expected evolution scenario to evolutionary suicide.
	
\begin{figure}[!ht]
\center
 \includegraphics[width=9cm]{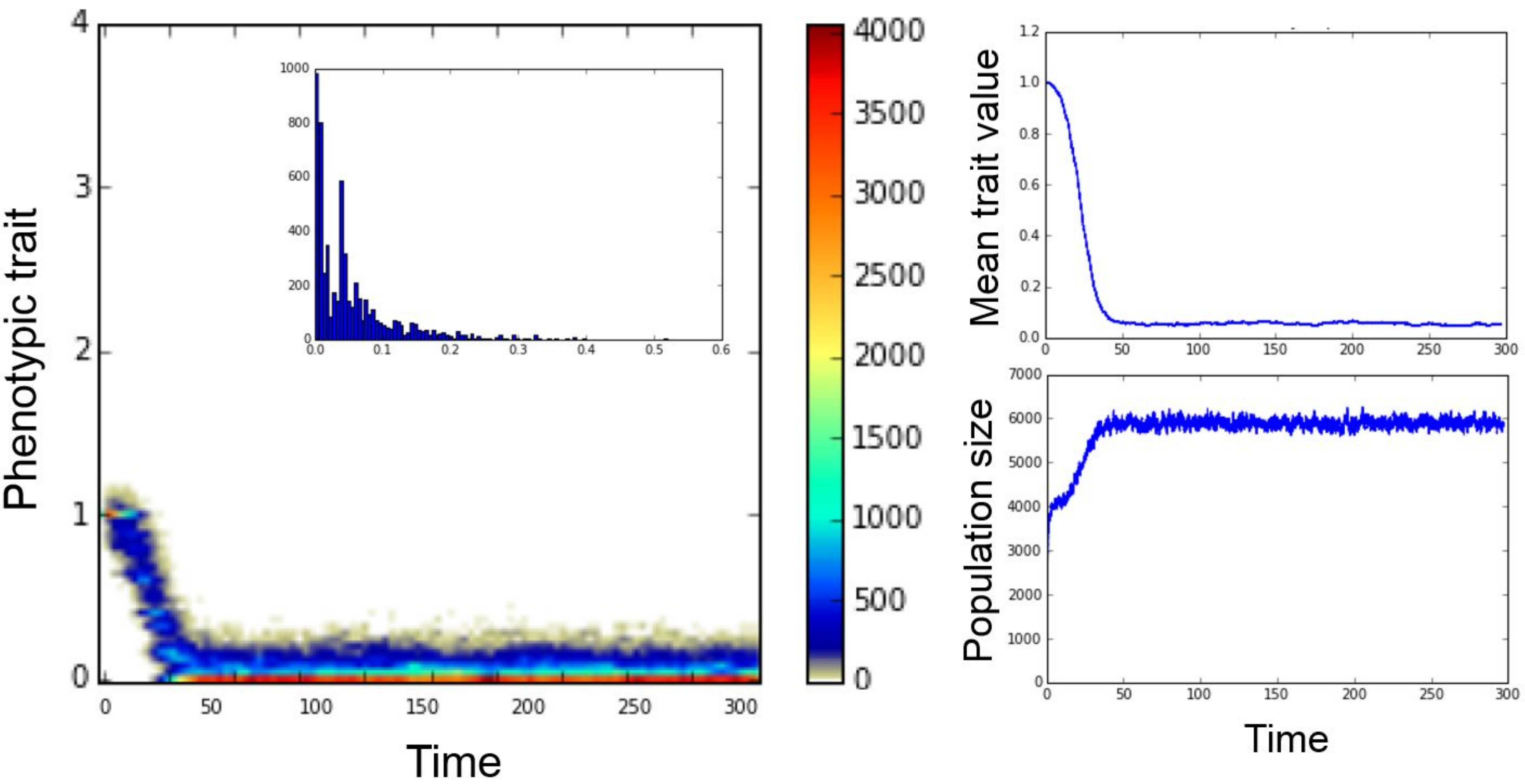}
\caption{{\small\textit{Simulations of eco-evolutionary dynamics with unilateral trait transfer. Transfer constant $\tau = 0$. Inset in main panel shows the trait distribution at the end of the simulation. }}}\label{fig:7.1}
 \end{figure}

\noindent The case $\tau = 0$ (Fig. \ref{fig:7.1}) is the null scenario without transfer. The evolution drives the population to its optimal trait  $\,0\, $ corresponding to a size at equilibrium equal to $ 1000\times {b(0)-d(0)\over C} = 6000$ individuals.

\begin{figure}[!ht]
\center
 \includegraphics[width=9cm]{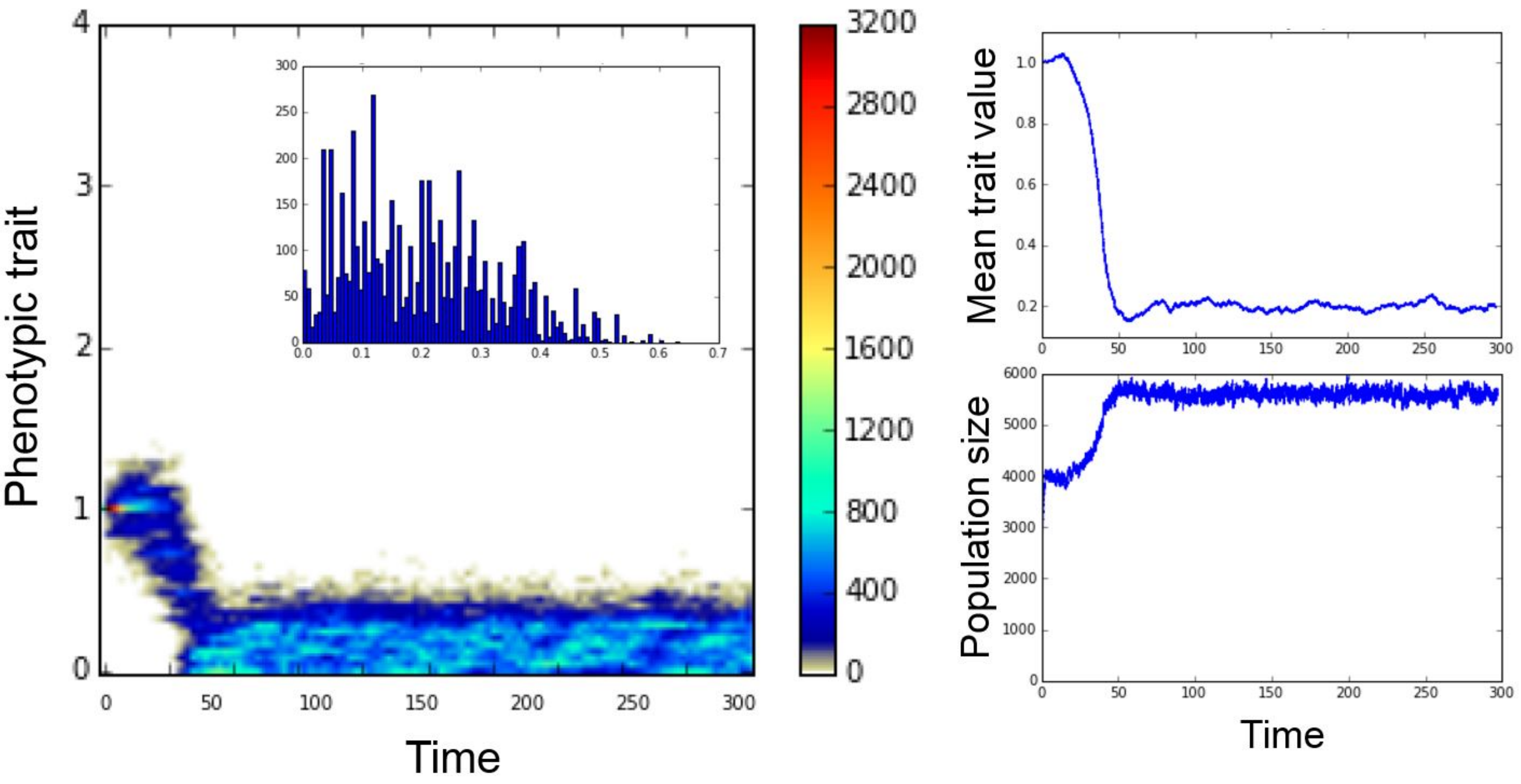}
\caption{{\small \textit{Simulations of eco-evolutionary dynamics with unilateral trait transfer. Transfer constant $\tau = 0.2$. Inset in main panel shows the trait distribution at the end of the simulation. }}}\label{fig:7.2}
 \end{figure}

\me The case $\tau=0.2$ (Fig. \ref{fig:7.2}) has characteristics similar to the case $\tau = 0$.

\begin{figure}[!ht]
\center
 \includegraphics[width=9cm]{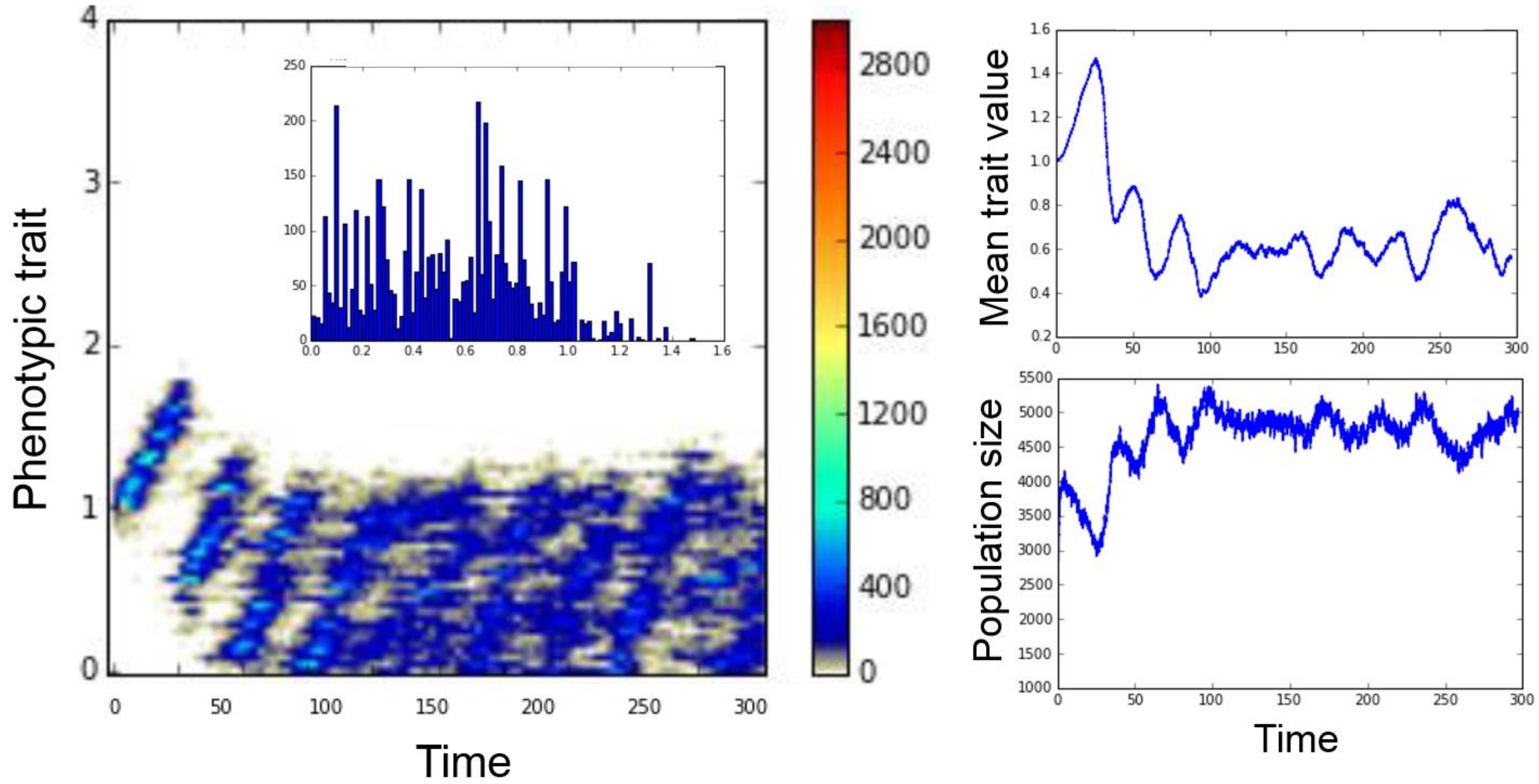}
\caption{{\small \textit{Simulations of eco-evolutionary dynamics with unilateral trait transfer. Transfer constant $\tau = 0.6$. This example illustrates a pattern of stepwise evolution caused by horizontal transfer. Inset in main panel shows the trait distribution at the end of the simulation. }}}\label{fig:7.3}
  \end{figure}

\me The evolution scenario in the case $\tau = 0.6$ (Fig. \ref{fig:7.3}) is rather different than the one for small $\tau$. High transfer  converts at first individuals to larger traits and in the same time the population decreases since for a given trait $\,x$, the equilibrium size $\,N_{eq}= {b(x)-d(x)\over C}\times 1000 = 2000(3-x)$. At some point, the population size is so small that the transfer doesn't play a role anymore leading to the
brutal resurgence of a quasi-invisible strain, issued from a few well adapted individuals with small traits.
 Computation shows that a  small trait $x_{small}$ can  invade the resident population with  trait $\overline{x}$ if $\,S(x_{small};\overline{x})=\overline{x}- x_{small} -\tau>0$.
  If such a mutant appears, it reproduces faster and its subpopulation immediately   kills the population with trait
$\overline{x}$.

\me Note that the successive resurgences drive the mean trait towards the optimal trait $0$.\\

\begin{figure}[!ht]
\center
\includegraphics[width=9cm]{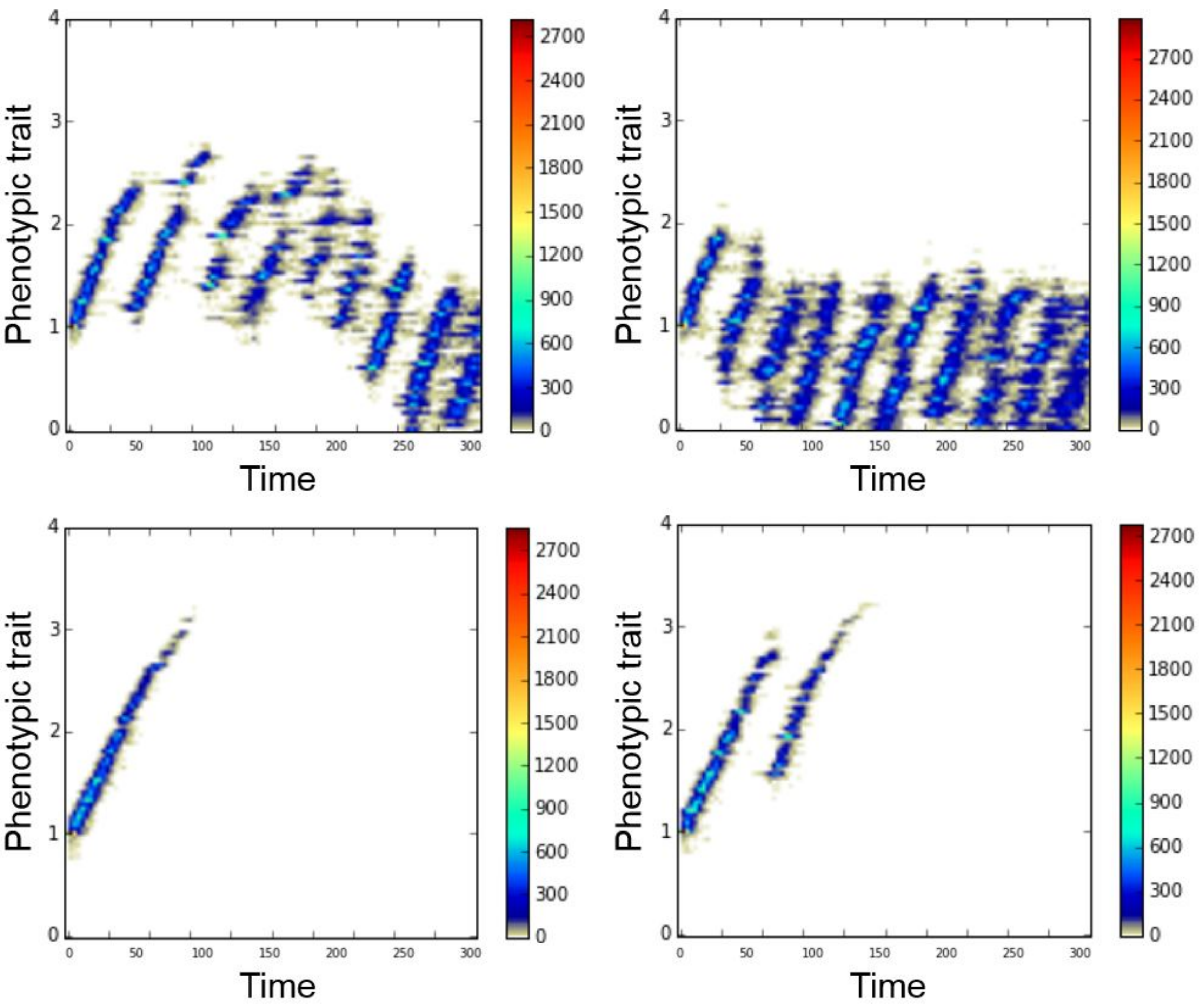}
\caption{{\small \textit{Simulations of eco-evolutionary dynamics with unilateral trait transfer. Transfer constant $\tau = 0.7$. Four simulation runs started with the same initial conditions illustrate random macroscopic evolution, including the possible occurrence of evolutionary suicide (lower panel). }}}\label{fig:7.4}
\end{figure}

\me Increasing further the transfer rate to $\tau=0.7$ (Fig. \ref{fig:7.4}), we can see either patterns as those above, with resurgences driving the mean trait towards the optimal trait, or extinctions of the population when there is no resurgence. The two simulations in the second line of Fig. \ref{fig:7.4} show evolutionary suicides: because $\tau$ is big, no small trait is left in these simulations to allow resurgence and the population reaches a state where the traits are so maladapted that the dies.

\begin{figure}[!ht]
\center
 \includegraphics[width=9cm]{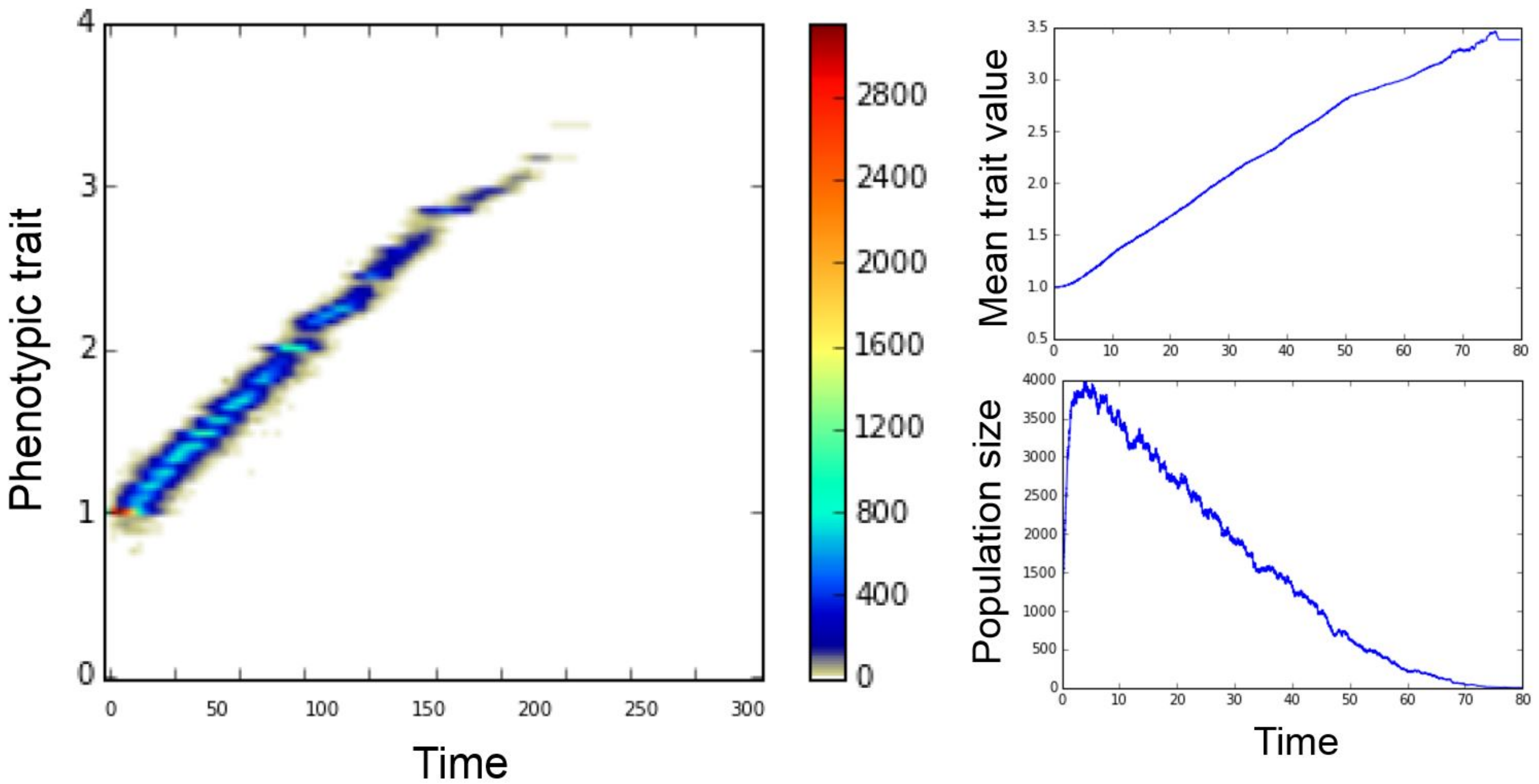}
\caption{{\small \textit{Simulations of eco-evolutionary dynamics with unilateral trait transfer. Transfer constant $\tau = 1.0$. Here evolution always leads to population extinction.}}}\label{fig:7.5}
\end{figure}

\me When $\tau=1$ (Fig. \ref{fig:7.5}), HGT impedes the population to keep a small mean trait to survive and we get evolutionary suicide in all the simulations that were done wih these parameters. The transfer drives the traits to larger and larger values, corresponding to lower and lower population sizes.

 \bi
 \emph{Acknowledgements:} S.B., S.M. and V.C.T. have been supported by  the Chair ``Mod\'elisation Math\'ematique et Biodiversit\'e" of Veolia Environnement-Ecole Polytechnique-Museum National d'Histoire Naturelle-Fondation X. V.C.T. also acknowledges support from Labex CEMPI (ANR-11-LABX-0007-01).

{\footnotesize
\providecommand{\noopsort}[1]{}\providecommand{\noopsort}[1]{}\providecommand{\noopsort}[1]{}\providecommand{\noopsort}[1]{}

}\end{document}